\documentclass[12pt]{amsart}%
\usepackage[T1]{fontenc}
\usepackage{amssymb}
\usepackage{a4wide}
\usepackage{graphicx}
\usepackage{verbatim}
\usepackage{amsmath}
\usepackage{version}
\usepackage{amsfonts}%
\providecommand{\U}[1]{\protect\rule{.1in}{.1in}}
\newtheorem{theorem}{Theorem}

\newtheorem{corollary}[theorem]{Corollary}

\newtheorem{lemma}[theorem]{Lemma}

\newtheorem{proposition}[theorem]{Proposition}
\newtheorem{remark}[theorem]{Remark}

\begin{document}

\title[Semiclassical measures on the torus]{Semiclassical measures for the Schr\"odinger equation on the torus}
\author{Nalini Anantharaman}
\address{Universit\'{e} Paris-Sud 11, Math\'{e}matiques, B\^{a}t. 425, 91405 ORSAY
CEDEX, FRANCE}
\email{Nalini.Anantharaman@math.u-psud.fr}
\author{Fabricio Maci\`a}
\address{Universidad Polit\'{e}cnica de Madrid. DCAIN, ETSI Navales. Avda. Arco de la
Victoria s/n. 28040 MADRID, SPAIN}
\email{Fabricio.Macia@upm.es}
\thanks{F. Maci{\`a} was supported by grants MTM2007-61755, MTM2010-16467 (MEC) and
Santander-Complutense 34/07-15844. N. Anantharaman wishes to acknowledge the
support of Agence Nationale de la Recherche, under the grant ANR-09-JCJC-0099-01.}

\begin{abstract}
{}In this article, the structure of semiclassical
measures for solutions to the linear Schr\"{o}dinger
equation on the torus is analysed. We show that the disintegration of
such a measure on every invariant lagrangian torus is absolutely
continuous with respect to the Lebesgue measure. We obtain an
expression of the Radon-Nikodym derivative in terms of the
sequence of initial data and show that it satisfies an explicit
propagation law. As a consequence, we also prove an observability inequality, saying that the $L^2$-norm of a solution on any open subset of the torus controls the full $L^2$-norm.

\end{abstract}
\maketitle

\newcommand{\nwc}{\newcommand}
\nwc{\nwt}{\newtheorem}
\nwt{coro}{Corollary}
\nwt{ex}{Example}
\nwt{prop}{Proposition}
\nwt{defin}{Definition}

%font change

\nwc{\mf}{\mathbf} %Latex (as in \bf not tilted math letters)
\nwc{\blds}{\boldsymbol} %Latex 
\nwc{\ml}{\mathcal} %Latex

%greek letters

\nwc{\lam}{\lambda}
\nwc{\del}{\delta}
\nwc{\Del}{\Delta}
\nwc{\Lam}{\Lambda}
\nwc{\elll}{\ell}
%blackboard bold math

\nwc{\IA}{\mathbb{A}} %algebraic
\nwc{\IB}{\mathbb{B}} %ball
\nwc{\IC}{\mathbb{C}} %complex
\nwc{\ID}{\mathbb{D}} %Dedekind
\nwc{\IE}{\mathbb{E}} %Euklides
\nwc{\IF}{\mathbb{F}} %finite field
\nwc{\IG}{\mathbb{G}} %Gauss
\nwc{\IH}{\mathbb{H}} %Hilbert\N-subgroup
\nwc{\IN}{\mathbb{N}} %natural
\nwc{\IP}{\mathbb{P}} %prime
\nwc{\IQ}{\mathbb{Q}} %rational
\nwc{\IR}{\mathbb{R}} %real
\nwc{\IS}{\mathbb{S}} %sphere
\nwc{\IT}{\mathbb{T}} %torus
\nwc{\IZ}{\mathbb{Z}} %integers
\def\bbbone{{\mathchoice {1\mskip-4mu {\rm{l}}} {1\mskip-4mu {\rm{l}}}
{ 1\mskip-4.5mu {\rm{l}}} { 1\mskip-5mu {\rm{l}}}}}
\def\bbleft{{\mathchoice {[\mskip-3mu {[}} {[\mskip-3mu {[}}{[\mskip-4mu {[}}{[\mskip-5mu {[}}}}
\def\bbright{{\mathchoice {]\mskip-3mu {]}} {]\mskip-3mu {]}}{]\mskip-4mu {]}}{]\mskip-5mu {]}}}}
\nwc{\setK}{\bbleft 1,K \bbright}
\nwc{\setN}{\bbleft 1,\cN \bbright}
 \newcommand{\Lim}{\mathop{\longrightarrow}\limits}
%Straight (vector) bold letters

%lowercase

\nwc{\va}{{\bf a}}
\nwc{\vb}{{\bf b}}
\nwc{\vc}{{\bf c}}
\nwc{\vd}{{\bf d}}
\nwc{\ve}{{\bf e}}
\nwc{\vf}{{\bf f}}
\nwc{\vg}{{\bf g}}
\nwc{\vh}{{\bf h}}
\nwc{\vi}{{\bf i}}
\nwc{\vI}{{\bf I}}
\nwc{\vj}{{\bf j}}
\nwc{\vk}{{\bf k}}
\nwc{\vl}{{\bf l}}
\nwc{\vm}{{\bf m}}
\nwc{\vM}{{\bf M}}
\nwc{\vn}{{\bf n}}
\nwc{\vo}{{\it o}}
\nwc{\vp}{{\bf p}}
\nwc{\vq}{{\bf q}}
\nwc{\vr}{{\bf r}}
\nwc{\vs}{{\bf s}}
\nwc{\vt}{{\bf t}}
\nwc{\vu}{{\bf u}}
\nwc{\vv}{{\bf v}}
\nwc{\vw}{{\bf w}}
\nwc{\vx}{{\bf x}}
\nwc{\vy}{{\bf y}}
\nwc{\vz}{{\bf z}}
\nwc{\bal}{\blds{\alpha}}
\nwc{\bep}{\blds{\epsilon}}
\nwc{\barbep}{\overline{\blds{\epsilon}}}
\nwc{\bnu}{\blds{\nu}}
\nwc{\bmu}{\blds{\mu}}
\nwc{\bet}{\blds{\eta}}

%bold letters
%\b* letters are tilted in math mode and scale in equations. 
%but cannot be used in plain text format.

%I. lowercase

\nwc{\bk}{\blds{k}}
\nwc{\bm}{\blds{m}}
\nwc{\bM}{\blds{M}}
\nwc{\bp}{\blds{p}}
\nwc{\bq}{\blds{q}}
\nwc{\bn}{\blds{n}}
\nwc{\bv}{\blds{v}}
\nwc{\bw}{\blds{w}}
\nwc{\bx}{\blds{x}}
\nwc{\bxi}{\blds{\xi}}
\nwc{\by}{\blds{y}}
\nwc{\bz}{\blds{z}}

%caligraphic

\nwc{\cA}{\ml{A}}
\nwc{\cB}{\ml{B}}
\nwc{\cC}{\ml{C}}
\nwc{\cD}{\ml{D}}
\nwc{\cE}{\ml{E}}
\nwc{\cF}{\ml{F}}
\nwc{\cG}{\ml{G}}
\nwc{\cH}{\ml{H}}
\nwc{\cI}{\ml{I}}
\nwc{\cJ}{\ml{J}}
\nwc{\cK}{\ml{K}}
\nwc{\cL}{\ml{L}}
\nwc{\cM}{\ml{M}}
\nwc{\cN}{\ml{N}}
\nwc{\cO}{\ml{O}}
\nwc{\cP}{\ml{P}}
\nwc{\cQ}{\ml{Q}}
\nwc{\cR}{\ml{R}}
\nwc{\cS}{\ml{S}}
\nwc{\cT}{\ml{T}}
\nwc{\cU}{\ml{U}}
\nwc{\cV}{\ml{V}}
\nwc{\cW}{\ml{W}}
\nwc{\cX}{\ml{X}}
\nwc{\cY}{\ml{Y}}
\nwc{\cZ}{\ml{Z}}

\nwc{\fA}{\mathfrak{a}}
\nwc{\fB}{\mathfrak{b}}
\nwc{\fC}{\mathfrak{c}}
\nwc{\fD}{\mathfrak{d}}
\nwc{\fE}{\mathfrak{e}}
\nwc{\fF}{\mathfrak{f}}
\nwc{\fG}{\mathfrak{g}}
\nwc{\fH}{\mathfrak{h}}
\nwc{\fI}{\mathfrak{i}}
\nwc{\fJ}{\mathfrak{j}}
\nwc{\fK}{\mathfrak{k}}
\nwc{\fL}{\mathfrak{l}}
\nwc{\fM}{\mathfrak{m}}
\nwc{\fN}{\mathfrak{n}}
\nwc{\fO}{\mathfrak{o}}
\nwc{\fP}{\mathfrak{p}}
\nwc{\fQ}{\mathfrak{q}}
\nwc{\fR}{\mathfrak{r}}
\nwc{\fS}{\mathfrak{s}}
\nwc{\fT}{\mathfrak{t}}
\nwc{\fU}{\mathfrak{u}}
\nwc{\fV}{\mathfrak{v}}
\nwc{\fW}{\mathfrak{w}}
\nwc{\fX}{\mathfrak{x}}
\nwc{\fY}{\mathfrak{y}}
\nwc{\fZ}{\mathfrak{z}}

%% (wide)tilde letters

\nwc{\tA}{\widetilde{A}}
\nwc{\tB}{\widetilde{B}}
\nwc{\tE}{E^{\vareps}}
%\nwc{\tcO}{\widetilde{\mathcal{O}}}
\nwc{\tk}{\tilde k}
\nwc{\tN}{\tilde N}
\nwc{\tP}{\widetilde{P}}
\nwc{\tQ}{\widetilde{Q}}
\nwc{\tR}{\widetilde{R}}
\nwc{\tV}{\widetilde{V}}
\nwc{\tW}{\widetilde{W}}
\nwc{\ty}{\tilde y}
\nwc{\teta}{\tilde \eta}
\nwc{\tdelta}{\tilde \delta}
\nwc{\tlambda}{\tilde \lambda}
%\nwc{\tchi}{\tilde \chi}
\nwc{\ttheta}{\tilde \theta}
\nwc{\tvartheta}{\tilde \vartheta}
\nwc{\tPhi}{\widetilde \Phi}
\nwc{\tpsi}{\tilde \psi}
\nwc{\tmu}{\tilde \mu}

%miscellany
\nwc{\To}{\longrightarrow} %limits

\nwc{\ad}{\rm ad}
\nwc{\eps}{\epsilon}
\nwc{\ep}{\epsilon}
\nwc{\vareps}{\varepsilon}

\def\bom{\mathbf{\omega}}
\def\om{{\omega}}
\def\ep{\epsilon}
\def\tr{{\rm tr}}
\def\diag{{\rm diag}}
\def\Tr{{\rm Tr}}
\def\i{{\rm i}}
\def\mi{{\rm i}}
\def\e{{\rm e}}
\def\sq2{\sqrt{2}}
\def\sqn{\sqrt{N}}
\def\vol{\mathrm{vol}}
\def\defi{\stackrel{\rm def}{=}}
\def\t2{{\mathbb T}^2}
%\def\tt2{{\mathbb T}^2}
%\nwc{\t1}{{\mathbb T}^1}
\def\s2{{\mathbb S}^2}
\def\hn{\mathcal{H}_{N}}
\def\shbar{\sqrt{\hbar}}
\def\A{\mathcal{A}}
\def\N{\mathbb{N}}
\def\T{\mathbb{T}}
\def\R{\mathbb{R}}
\def\RR{\mathbb{R}}
\def\Z{\mathbb{Z}}
\def\C{\mathbb{C}}
\def\O{\mathcal{O}}
\def\Sp{\mathcal{S}_+}
\def\Lap{\triangle}
\nwc{\lap}{\bigtriangleup}
\nwc{\rest}{\restriction}
\nwc{\Diff}{\operatorname{Diff}}
\nwc{\diam}{\operatorname{diam}}
\nwc{\Res}{\operatorname{Res}}
\nwc{\Spec}{\operatorname{Spec}}
\nwc{\Vol}{\operatorname{Vol}}
\nwc{\Op}{\operatorname{Op}}
\nwc{\supp}{\operatorname{supp}}
\nwc{\Span}{\operatorname{span}}

\nwc{\dia}{\varepsilon}
\nwc{\cut}{f}
\nwc{\qm}{u_\hbar}

\def\hto0{\xrightarrow{\hbar\to 0}}
\def\htoo{\stackrel{h\to 0}{\longrightarrow}}
\def\rto0{\xrightarrow{r\to 0}}
\def\rtoo{\stackrel{r\to 0}{\longrightarrow}}
\def\ntoinf{\xrightarrow{n\to +\infty}}

\providecommand{\abs}[1]{\lvert#1\rvert}
\providecommand{\norm}[1]{\lVert#1\rVert}
\providecommand{\set}[1]{\left\{#1\right\}}

\nwc{\la}{\langle}
\nwc{\ra}{\rangle}
\nwc{\lp}{\left(}
\nwc{\rp}{\right)}

%\nwc{\bal}{\begin{align}}
\nwc{\bequ}{\begin{equation}}
\nwc{\be}{\begin{equation}}
\nwc{\ben}{\begin{equation*}}
\nwc{\bea}{\begin{eqnarray}}
\nwc{\bean}{\begin{eqnarray*}}
\nwc{\bit}{\begin{itemize}}
\nwc{\bver}{\begin{verbatim}}

%\nwc{\eal}{\end{align}}
\nwc{\eequ}{\end{equation}}
\nwc{\ee}{\end{equation}}
\nwc{\een}{\end{equation*}}
\nwc{\eea}{\end{eqnarray}}
\nwc{\eean}{\end{eqnarray*}}
\nwc{\eit}{\end{itemize}}
\nwc{\ever}{\end{verbatim}}

\newcommand{\defeq}{\stackrel{\rm{def}}{=}}

\section{Introduction}

Consider the torus $\mathbb{T}^{d}:=\left(
\mathbb{R/}2\pi\mathbb{Z}\right) ^{d}$ equipped with the standard
flat metric. We denote by $\Delta$ the associated Laplacian. We
are interested in understanding dynamical properties related to
propagation of singularities by the (time-dependent) linear
Schr\"{o}dinger equation
\[
i\frac{\partial u}{\partial t}(t,x)=\left( -\frac{1}{2}\Delta+V(t,
x)\right) u(t,x),\qquad u\rceil_{t=0}=u_0\in L^{2}(\mathbb{T}^{d}).
\]
More precisely, given a sequence of initial conditions $u_{n}\in
L^{2}(\mathbb{T}^{d})$, we shall investigate the regularity properties of the \emph{Wigner distributions }and \emph{semiclassical measures} associated with $u_n(t, x)$.
These describe how the $L^{2}$-norm is distributed in the cotangent bundle
$T^{\ast}\mathbb{T}^{d}=\mathbb{T}^{d}\times\mathbb{R}^{d}$ (position $\times$ frequency). Our main results, Theorems
\ref{t:main} and \ref{t:precise} below, provide a description of the
regularity properties and, more generally, the global structure of
semiclassical measures associated to sequences of solutions to the
Schr\"{o}dinger equation.

These results are aimed to give a description of the
high-frequency behavior of the linear Schr\"{o}dinger flow. This
aspect of the dynamics is particularly relevant in the study of
the quantum-classical correspondence principle, but is also
related to other dynamical properties such as dispersion and
unique continuation (see the discussion below and the articles
\cite{MaciaAv, MaciaDispersion, AnantharamanMaciaSurv} for a more
precise account and detailed references on these issues). As a
corollary of Theorem \ref{t:precise}, we prove an observability
inequality on any open subset of the torus, for the Schr\"odinger
equation with a time-independent potential~: Theorem
\ref{t:obs}.\medskip

We assume the following regularity condition on the potential
$V\in L^{\infty}\left(  \mathbb{R\times T}^{d}\right)  $~:

\smallskip {\bf (R)} For every $T>0$, for every $\eps>0$, there
exists a {\em compact} set $K_{\eps}\subset  [0, T]\times \T^d$,
of Lebesgue measure $<\eps$, and $V_\eps\in C([0, T]\times \T^d)$,
such that  $\left|V-V_\eps\right|\leq \eps$ on  $\left([0,
T]\times \T^d\right)\setminus K_{\eps}$.
%the map $t\mapsto V(t)$ is
%continuously differentiable with values in $L^{\infty}(\T^{d})\spadesuit$. The
%regularity with respect to $t$ is used to invoke \cite[Theorem X.71]{RSII} to
%have the existence of a unitary propagator $U_{V}(t^{\prime},t)$ ($t^{\prime
%},t\in\R$) such that $u_{h}(t)=U_{V}(t^{\prime},t)u_{h}(t^{\prime})$. This
%regularity in $t$ will not be used elsewhere in the paper.
\smallskip

We believe that this assumption should not be necessary. In any case, assumption (R) already covers a broad class of examples.

% \texttt{[I guess we
%can invoke Duhamel's formula for the existence of solutions. Nalini : I don't see what Fabricio means.]}
We shall focus on the propagator starting at time
$0$, denoted by $U_{V}(t)$; {\em i.e.} $u(t)=U_{V}(t)u_0$.

Let us define the notion of Wigner distribution. We will use the
semiclassical point of view, and denote by $(u_{h})$ our family of
initial conditions, where $h>0$ is a real parameter going to $0$. The parameter $h$ acts as a scaling factor
on the frequencies, and the limit $h\To0^{+}$ corresponds to the high-frequency regime.
We will always assume that the functions $u_{h}$ are normalized in
$L^{2}(\mathbb{T}^{d})$. The \emph{Wigner distribution} associated
to $u_{h}$ (at scale $h$) is a distribution on the cotangent
bundle $T^{\ast}\mathbb{T}^{d}$, defined by
\[
\int_{T^{\ast}\mathbb{T}^{d}}a(x,\xi)w_{u_{h}}^{h}(dx,d\xi)=\left\langle
u_{h},\Op_{h}(a)u_{h}\right\rangle _{L^{2}(\mathbb{T}^{d})},\qquad
\mbox{ for all }a\in C_{c}^{\infty}(T^{\ast}\mathbb{T}^{d}),
\]
where $\Op_{h}(a)$ is the operator on $L^{2}(\mathbb{T}^{d})$ associated to
$a$ by the Weyl quantization (Section \ref{s:PDO}). More explicitly, we have
\[
\int_{T^{\ast}\mathbb{T}^{d}}a(x,\xi)w_{u_{h}}^{h}(dx,d\xi)=\frac{1}{\left(
2\pi\right)  ^{d/2}}\sum_{k,j\in\Z^{d}}\hat{u}_{h}(k)\overline{\hat{u}_{h}%
(j)}\hat{a}_{j-k}\left(  \frac{h}{2}(k+j)\right)  ,
\]
where $\hat{u}_{h}(k):=\int_{\mathbb{T}^{d}}u_{h}(x)\frac{e^{-ik.x}}%
{(2\pi)^{d/2}}dx$ and $\hat{a}_{k}(\xi):=\int_{\mathbb{T}^{d}}a(x,\xi
)\frac{e^{-ik.x}}{(2\pi)^{d/2}}dx$ denote the respective Fourier coefficients
of $u_{h}$ and $a$, with respect to the variable $x\in\mathbb{T}^{d}$. We note
that, if $a$ is a function on $T^{\ast}\mathbb{T}^{d}=\mathbb{T}^{d}%
\times\R^{d}$ that depends only on the first coordinate, then
\begin{equation}
\int_{T^{\ast}\mathbb{T}^{d}}a(x)w_{u_{h}}^{h}(dx,d\xi)=\int_{\mathbb{T}^{d}%
}a(x)|u_{h}(x)|^{2}dx. \label{e:proj}%
\end{equation}

The main object of our study will be the Wigner distributions $w_{U_{V}%
(t)u_{h}}^{h}$. When no confusion arises, we will more simply denote them by
$w_{h}(t,\cdot)$. By standard estimates on the norm of $\Op_{h}(a)$ (the
Calder\'{o}n-Vaillancourt theorem, section \ref{s:PDO}), $t\mapsto
w_{h}(t,\cdot)$ belongs to $L^{\infty}(\R;\cD^{\prime}\left(  T^{\ast
}\mathbb{T}^{d}\right)  )$, and is uniformly bounded in that space as
$h\To0^{+}$. Thus, one can extract subsequences that converge in the
weak-$\ast$ topology on $L^{\infty}(\R;\cD^{\prime}\left(  T^{\ast}%
\mathbb{T}^{d}\right)  )$. In other words, after possibly extracting a
subsequence, we have
\[
\int_{\R}\varphi(t)a(x,\xi)w_{h}(t,dx,d\xi)dt\Lim_{h\To0}\int_{\R}%
\varphi(t)a(x,\xi)\mu(t,dx,d\xi)dt
\]
for all $\varphi\in L^{1}(\R)$ and $a\in C_{c}^{\infty}(T^{\ast}\mathbb{T}%
^{d}).$ It also follows from standard properties of the Weyl quantization that
the limit $\mu$ has the following properties~:

\begin{itemize}
\item $\mu\in L^{\infty}(\R;\cM_{+}(T^{\ast}\mathbb{T}^{d}))$, meaning that
for almost all $t$, $\mu(t,\cdot)$ is a positive measure on $T^{\ast
}\mathbb{T}^{d}$.

\item The unitary character of $U_{V}(t)$ implies that $\int_{T^{\ast
}\mathbb{T}^{d}}\mu(t,dx,d\xi)$ does not depend on $t$; from the normalization
of $u_{h}$, we have $\int_{T^{\ast}\mathbb{T}^{d}}\mu(t,dx,d\xi)\leq1$, the
inequality coming from the fact that $T^{\ast}\mathbb{T}^{d}$ is not compact,
and that there may be an escape of mass to infinity.

\item Define the geodesic flow $\phi_{\tau}:T^{\ast}\mathbb{T}^{d}\To
T^{\ast}\mathbb{T}^{d}$ by $\phi_{\tau}(x,\xi):=(x+\tau\xi,\xi)$ ($\tau\in
\R$). The Weyl quantization enjoys the following property~:
\begin{equation}
\left[-\frac12\Delta,
\Op_{h}(a)\right]=\frac1{ih}\Op_{h}(\xi\cdot\partial_x a).
\label{e:weylcomm}
\end{equation}
This implies that $\mu(t,\cdot)$ is invariant under $\phi_{\tau}$,
for almost all $t$ and all $\tau\in\R$ (the argument is recalled
in Lemma \ref{Lemma Inv}).
%; in addition, if $\mu_{0}$ denotes any weak limit of
%$\left(  w_{u_{h}}^{h}\right)  $ then $\bar{\mu}=\int_{\mathbb{R}^{d}}\mu
%0}\left(  \cdot,d\xi\right)  $ (see remark 24 in \cite{MaciaAv})

\end{itemize}

We refer to \cite{MaciaAv} for details. We can now state our first main
result, which deals with the regularity properties of the measures $\mu$.

\begin{theorem}
\label{t:main}(i) Let $\mu$ be a weak-$\ast$ limit of the family $w_{h}$.
Then, for almost all $t$, $\int_{\R^{d}}\mu(t,\cdot,d\xi)$ is an absolutely
continuous measure on $\mathbb{T}^{d}$.

(ii) In fact, the following stronger statement holds.
Let $\bar{\mu}$ be the measure on $\mathbb{R}^{d}$ image of $\mu
(t,\cdot)$ under the projection map $(x,\xi)\mapsto\xi$. Then $\bar{\mu}$ does
not depend on $t$.

For every bounded
measurable function $f$, and every $L^{1}$-function $\theta(t)$ write
\[
\int_{\mathbb{R}}\int_{\mathbb{T}^{d}\times\mathbb{R}^{d}}f(x,\xi)\mu(t,dx,d\xi)\theta
(t)dt=\int_{\mathbb{R}}\int_{\mathbb{R}^{d}}\left(
\int_{\mathbb{T}^{d}}f(x,\xi)\mu_{\xi }(t,dx)\right)
\bar{\mu}(d\xi)\theta(t)dt,
\]
where $\mu_{\xi}(t,\cdot)$ is the disintegration\footnote{When $\mu\left(
t,\cdot\right)  $ is a probability measure, $\mu_{\xi}(t,\cdot)$ is the
conditional law of $x$ knowing $\xi$, when the pair $(x,\xi)$ is distributed
according to $\mu(t,\cdot)$} of $\mu(t,\cdot)$ with respect to the variable
$\xi$. Then for $\bar{\mu}$-almost every $\xi$, the measure $\mu_{\xi}%
(t,\cdot)$ is absolutely continuous.
\end{theorem}

The first assertion in Theorem
\ref{t:main} may be restated in a simpler, concise way.

\begin{corollary}
\label{t:example}Let $(u_{n})$ be a sequence in $L^{2}(\mathbb{T}^{d})$, such
that $\norm{u_n}_{L^{2}(\mathbb{T}^{d})}=1$ for all $n$. Consider the sequence
of probability measures $\nu_{n}$ on $\mathbb{T}^{d}$, defined by
\begin{equation}
\nu_{n}(dx)=\left(  \int_{0}^{1}|U_{V}(t)u_{n}(x)|^{2}dt\right)  dx.
\label{probmes}%
\end{equation}
Let $\nu$ be any weak-$\ast$ limit of the sequence $(\nu_{n})$~: then $\nu$ is
absolutely continuous.
\end{corollary}

Our next result enlightens the structure of the set of semiclassical
measures arising as weak-$\ast$ limits of sequences $\left(  w_{h}\right)  $.
It gives a description of the Radon-Nikodym derivatives of the measures
$\int_{\R^{d}}\mu(t,\cdot,d\xi)$ and clarifies the link between $\mu(0,\cdot)$
and $\mu(t,\cdot)$. It was already noted in \cite{MaciaAv} (in the case $V=0$)
that the dependence of $\mu(t,\cdot)$ on the sequence of initial conditions is
a subtle issue~: although $w_{h}(0,\cdot)=w_{u_{h}}^{h}$ completely determines
$w_{h}(t,\cdot)=w_{U_{V}(t)u_{h}}^{h}$ for all $t$, it is not true that the
weak-$\ast$ limits of $w_{h}(0,\cdot)$ determine $\mu(t,\cdot)$ for all $t$.
In \cite{MaciaAv}, one can find examples of two sequences $(u_{h})$ and
$(v_{h})$ of initial conditions, such that $w_{u_{h}}^{h}$ and $w_{v_{h}}^{h}$
have the same limit in $\cD^{\prime}(T^{\ast}\mathbb{T}^{d})$, but
$w_{U_{V}(t)u_{h}}^{h}$ and $w_{U_{V}(t)v_{h}}^{h}$ have different limits in
$L^{\infty}(\R;\cD^{\prime}(T^{\ast}\mathbb{T}^{d}))$.

In order to state Theorem \ref{t:precise}, we must introduce some notation. We
call a submodule $\Lambda\subset\Z^{d}$ primitive if $\left\langle
\Lambda\right\rangle \cap\mathbb{Z}^{d}=\Lambda$ (where $\left\langle
\Lambda\right\rangle $ denotes the linear subspace of $\mathbb{R}^{d}$ spanned
by $\Lambda$). If $b$ is a function on $\mathbb{T}^{d}$, let $\widehat{b}_{k}$,
$k\in\mathbb{Z}^{d}$, denote the Fourier coefficients of $b$. If $\widehat{b}_{k}=0$ for $k\not \in \Lambda$, we will say that $b$ has only Fourier modes
in $\Lambda$. This means that $b$ is constant in the directions orthogonal to
$\left\langle \Lambda\right\rangle $. Let $L_{\Lambda}^{p}(\IT^{d})$ denote
the subspace of $L^{p}\left(  \mathbb{T}^{d}\right)  $ consisting of functions
with Fourier modes in $\Lambda$.  If $b\in L^2(\T^d)$, we denote by
$\left\langle b\right\rangle _{\Lambda}$ its orthogonal projection onto $L_{\Lambda}^{2}(\IT^{d})$, in other words, the average of $b$ along
$\Lambda^{\bot}$~:%
\[
\left\langle b\right\rangle _{\Lambda}\left(  x\right)  :=\sum_{k\in\Lambda
}\widehat{b}_{k}\left(  t\right)  \frac{e^{ik\cdot x}}{\left(  2\pi\right)
^{d/2}}.
\]
Given $b\in L_{\Lambda}^{\infty}\left(
\mathbb{T}^{d}\right)  $, we will denote by $m_{b}$ the multiplication
operator by $b$, acting on $L_{\Lambda}^{2}(\IT^{d})$.

Finally, we denote by $U_{\la V\ra_{\Lambda}}(t)$ the unitary
propagator of the equation
\[
i\frac{\partial v}{\partial t}(t,x)=\left(  -\frac{1}{2}\Delta+\la
V\ra_{\Lambda }(t,x)\right)  v(t,x),\qquad v\rceil_{t=0}\in
L_{\Lambda}^{2}(\mathbb{T}^{d}).
\]

\begin{theorem}
\label{t:precise} For any sequence $(u_{h})$, we can extract a subsequence
such that the following hold~:

\begin{itemize}
\item the subsequence $w_{h}(t,\cdot)$ converges weakly-$\ast$ to a limit
$\mu(t,\cdot)$;

\item for each primitive submodule $\Lambda\subset\Z^{d}$, we can build from
the sequence of initial conditions $(u_{h})$ a nonnegative trace class
operator $\sigma_{\Lambda}$, acting on $L_{\Lambda}^{2}(\IT^{d})$%
;\footnote{This means that the integral kernel of
$\sigma_{\Lambda}$ is constant in the directions orthogonal to $\Lambda$.}

\item for almost all $t$, we have
\[
\int_{\R^{d}}\mu(t,\cdot,d\xi)=\sum_{\Lambda}\nu_{\Lambda}(t,\cdot),
\]
where $\nu_{\Lambda}(t,\cdot)$ is the measure on $\mathbb{T}^{d}$, whose
non-vanishing Fourier modes correspond to frequencies in $\Lambda$, defined
by
\[
\int_{\mathbb{T}^{d}}b(x)\nu_{\Lambda}(t,dx)=\Tr\left(  m_{\left\langle
b\right\rangle _{\Lambda}}\,U_{\la V\ra_{\Lambda}}(t)\,\sigma_{\Lambda
}\,{U_{\la V\ra_{\Lambda}}(t)}^{\ast}\right)  ,
\]
if $b\in L^{\infty}\left(  \mathbb{T}^{d}\right)  $.
\end{itemize}
\end{theorem}

Theorem \ref{t:precise} tells us more about the dependence of
$\mu(t,\cdot)$ with respect to $t$. If two sequences of initial
conditions $(u_{h})$ and $(v_{h})$ give rise to the same family of
operators $\sigma_{\Lambda}$, then they also give rise to the same
limit $\mu(t,\cdot)$. There are cases in which the measures
$\nu_{\Lambda}$ can be determined from the semiclassical measure
$\mu(0,\cdot)$ of the sequence of initial data~: in Corollary
\ref{c:vuLzero} in Section \ref{endthm1} we show that if
$\mu(0,\mathbb{T}^d\times\Lambda^{\bot})=0$ then $\nu_{\Lambda}$
vanishes identically.

Technically speaking, the operators $\sigma_{\Lambda}$ are built
in terms of $2$-microlocal semiclassical measures, that describe
how the sequences $\left( u_{h}\right)  $ concentrate along
certain coisotropic manifolds in phase-space. The technical construction of
$\sigma_{\Lambda}$ will only be achieved at the end of Section
\ref{s:propagation}.

We shall prove, as a consequence of Theorem \ref{t:precise}, the following result:

\begin{theorem}
\label{t:obs}Suppose $V\in L^\infty(\T^d)$ does not depend on time and satisfies condition (R). Then for every open set $\omega\subset
\mathbb{T}^{d}$ and every $T>0$ there exists a constant $C=C(T, \omega)>0$ such that:%
\begin{equation}
\left\Vert u_{0}\right\Vert _{L^{2}\left(  \mathbb{T}^{d}\right)  }^{2}\leq
C\int_{0}^{T}\left\Vert U_{V}\left(  t\right)  u_{0}\right\Vert _{L^{2}\left(
\omega\right)  }^{2}dt, \label{e:oi}%
\end{equation}
for every initial datum $u_{0}\in L^{2}\left(  \mathbb{T}^{d}\right)  $.
\end{theorem}

Note that this result implies the unique continuation property
for the Schr\"{o}dinger propagator $U_{V}$ from any open set $\left(
0,T\right)  \times\omega$. In other words, if $U_{V}\left(  t\right)  u_{0}=0$ on $\omega$ for all $t\in[0, T]$, then $u_0=0$.  Estimate (\ref{e:oi}) is usually known as an
$\emph{observability}$ \emph{inequality}; these type of estimates are
especially relevant in Control Theory (see \cite{LionsSurvey88}).

\bigskip

As a consequence of this result, with the notation of Theorem \ref{t:main} (ii), we deduce the following~:

\begin{corollary} \label{c:positiveopen}For $\bar\mu$-almost every $\xi$, we have
$$\int_0^T \mu_\xi(t, \omega)dt\geq \frac{T}{C(T, \omega)}.$$
\end{corollary}
This lower bound is uniform w.r.t. the initial data $u_h$ and to $\xi$.

\bigskip
\noindent\textbf{Relations to other work. } In the case $V=0$,
Corollary \ref{t:example} and the first assertion in Theorem
\ref{t:main} have been obtained by Zygmund \cite{Zygmund74} in the
case $d=1$. In the final remark of \cite{BourgainQL97}, Bourgain indicates a proof in arbitrary dimension, using fine properties of the distribution of lattice
points on paraboloids. When the sequence $(u_{n})$ consists of
eigenfunctions of $\Delta$ ($\nu_{n}(dx)=|u_{n}(x)|^{2}dx$, in
that case), the conclusion of Corollary \ref{t:example} was proved
by Zygmund ($d=2$), Bourgain (no restriction on $d$) and precised
in terms of regularity by Jakobson in \cite{JakobsonTori97}, by
studying the distribution of lattice points on ellipsoids. More
results on the regularity of $\mu$ can be found in \cite{Aissiou,
CookeTorus2d, JakNadToth01, Mockenhaupt96}.

Our methods are very different, and there is no obvious adaptation
of the technique of \cite{BourgainQL97,JakobsonTori97} to the case
$V\not =0$. Theorem \ref{t:precise} was proved in dimension $d=2$
for $V=0$ in \cite{MaciaTorus} using semiclassical methods, and we
develop and refine the ideas therein. We use in a decisive way the
dynamics of the geodesic flow (since we are on a flat torus, the
geodesic flow is a completely explicit object), and we use the
decomposition of the momentum space into resonant vectors of
various orders. The other main ingredient is the two-microlocal
calculus, in the spirit of the developments by Nier
\cite{NierScat} and Fermanian-Kammerer \cite{FermanianShocks,
Fermanian2micro}, and also \cite{MillerThesis,
FermanianGerardCroisements}. Our proof is written on the
\textquotedblleft square\textquotedblright\ torus. More precisely,
the property of the lattice $\Gamma=\Z^d\subset\R^d$ and of the
scalar product $\la \cdot, \cdot\ra$ (principal symbol of the
laplacian) that we use is that $\left[\la x, y\ra\in\IQ\: \forall
y\in\IQ\Gamma\Leftrightarrow x\in \IQ\Gamma\right]$. This
assumption can be removed and the results can be adapted to more
general lattices, but this requires a slightly different
presentation, that will appear in the work \cite{AFM}. Moreover,
it seems reasonable to think that Theorems \ref{t:main} and
\ref{t:precise} can be extended to more general completely
integrable systems and their quantizations \cite{AFM}. The
generalized statement would be that the disintegration of the
limit measure on regular lagrangian tori is absolutely continuous,
with respect to the Lebesgue measure on these tori.

Theorem \ref{t:obs} was first established by Jaffard
\cite{JaffardPlaques} in the case $V=0$ using techniques based on
the theory of lacunary Fourier series developed by Kahane. Since
then, several proofs of this result based on microlocal methods
and semiclassical measures (still for $V=0$) are available
\cite{BurqZworski03Billiards, MarzuolaBilliards, MaciaDispersion}.
Our proof of Theorem \ref{t:obs} will follow the lines of that
given in \cite{MaciaDispersion} and is based on the structure and
propagation result for semiclassical measures obtained in Theorem
\ref{t:precise}. At the same time as this paper was being written,
Burq and Zworski \cite{BurqZworski11} have given a proof of
Theorem \ref{t:obs} in the case $V\in C\left(
\mathbb{T}^{2}\right)$, which is an adaptation of their previous
work \cite{BurqZworski03Billiards}. Here, we exploit our results
about the structure of semiclassical measures to avoid the
semiclassical normal form argument (Burq and Zworski's
Propositions 2.5 and 2.10) and to lower the regularity of the
potential.

Corollary \ref{c:positiveopen} implies Corollary 4 of the article
by Wunsch \cite{Wunsch10} (which is expressed in terms of
wavefront sets) and holds in arbitrary dimension whereas Wunsch's
method is restricted to $d=2$.
\bigskip

\noindent\textbf{Acknowledgement. }Much of this work was done
while the second author was visiting the D\'{e}partement de
Math\'{e}matiques at Universit\'{e} Paris-Sud, in fall-winter
2009. He wishes to thank this institution for its kind
hospitality.

\section{Decomposition of an invariant measure on the torus \label{s:decompo}}
Before we start our construction in \S \ref{s:second}, we recall a few basic facts on the geodesic flow and its invariant measures.

Denote by $\mathcal{L}$ the family of all submodules $\Lambda$ of
$\mathbb{Z}^{d}$ which are \emph{primitive}, in the sense that
$\left\langle \Lambda\right\rangle \cap\mathbb{Z}^{d}=\Lambda$
(where $\left\langle \Lambda\right\rangle $ denotes the linear
subspace of $\mathbb{R}^{d}$ spanned
by $\Lambda$). For each $\Lambda\in\mathcal{L}$, we define%
\[
\Lambda^{\perp}:=\left\{  \xi\in\mathbb{R}^{d}:\xi\cdot
k=0\text{,\quad }\forall k\in\Lambda\right\}  ,
\]%
\[
\mathbb{T}_{\Lambda}:=\left\langle \Lambda\right\rangle
/2\pi\Lambda.
\]
Note that $\mathbb{T}_{\Lambda}$ is a submanifold of
$\mathbb{T}^{d}$ diffeomorphic to a torus of dimension
$\operatorname*{rk}\Lambda$. Its
cotangent bundle $T^{\ast}\mathbb{T}_{\Lambda}$ is $\mathbb{T}_{\Lambda}%
\times\left\langle \Lambda\right\rangle $. We shall use the
notation $\mathbb{T}_{\Lambda^{\perp}}$ to refer to the torus
$\Lambda^{\perp}/\left(
2\pi\mathbb{Z}^{d}\cap\Lambda^{\perp}\right)  $. Denote by $\Omega_{j}%
\subset\mathbb{R}^{d}$, for $j=0,...,d$, the set of resonant
vectors of order
exactly $j$, that is:%
\[
\Omega_{j}:=\left\{
\xi\in\mathbb{R}^{d}:\operatorname*{rk}\Lambda_{\xi }=d-j\right\}
,
\]
where $\Lambda_{\xi}:=\left\{
k\in\mathbb{Z}^{d}:k\cdot\xi=0\right\}  $. Note that the sets
$\Omega_{j}$ form a partition of $\mathbb{R}^{d}$, and that
$\Omega_{0}=\left\{  0\right\}  $; more generally,
$\xi\in\Omega_{j}$ if and only if the geodesic issued from any
$x\in\mathbb{T}^{d}$ in the direction $\xi$ is dense in a subtorus
of $\mathbb{T}^{d}$ of dimension $j$. The set
$\Omega:=\bigcup_{j=0}^{d-1}\Omega_{j}$ is usually called the set
of \emph{resonant }directions, whereas
$\Omega_{d}=\mathbb{R}^{d}\setminus\Omega$
is referred to as the set of \emph{non-resonant }vectors. Finally, write%
\[
R_{\Lambda}:=\Lambda^{\perp}\cap\Omega_{d-\operatorname*{rk}\Lambda}.
\]

The relevance of these definitions to the study of the geodesic flow is explained by the following remark.
Saying that $\xi\in R_{\Lambda}$ is equivalent to saying that (for
any
$x_{0}\in\IT^{d}$) the time-average $\frac{1}{T}\int_{0}^{T}\delta_{x_{0}%
+t\xi}\left(  x\right)  dt$ converges weakly to the Haar measure
on the torus $x_{0}+\mathbb{T}_{\Lambda^{\perp}}$, as
$T\rightarrow\infty\text{.}$

By construction, for $\xi\in
R_{\Lambda}$ we have $\Lambda_{\xi}=\Lambda$; moreover, if
$\operatorname*{rk}\Lambda=d-1$ then $R_{\Lambda}=\Lambda^{\perp
}\setminus\{0\}$. Finally,
\begin{equation}
\mathbb{R}^{d}=\bigsqcup_{\Lambda\in\mathcal{L}}R_{\Lambda}, \label{part}%
\end{equation}
that is, the sets $R_{\Lambda}$ form a partition of
$\mathbb{R}^{d}$. As a consequence, the following result holds.

\begin{lemma}
\label{Lemma decomposition}Let $\mu$ be a finite, positive Radon
measure\footnote{We denote by $\mathcal{M}_{+}\left(  T^{\ast}\mathbb{T}%
^{d}\right)  $ the set of all such measures.} on
$T^{\ast}\mathbb{T}^{d}$.
Then $\mu$ decomposes as a sum of positive measures:%
\begin{equation}
\mu=\sum_{\Lambda\in\mathcal{L}}\mu\rceil_{\mathbb{T}^{d}\times
R_{\Lambda}}.
\label{dec}%
\end{equation}

\end{lemma}

Given any $\mu\in\mathcal{M}_{+}\left(
T^{\ast}\mathbb{T}^{d}\right)  $ we define the Fourier
coefficients of $\mu$ as the complex measures on
$\mathbb{R}^{d}$:%
\[
\widehat{\mu}\left(  k,\cdot\right)
:=\int_{\mathbb{T}^{d}}\frac{e^{-ik\cdot x}}{\left(  2\pi\right)
^{d/2}}\mu\left(  dx,\cdot\right)  ,\quad k\in\mathbb{Z}^{d}.
\]
One has, in the sense of distributions,%
\[
\mu\left(  x,\xi\right)
=\sum_{k\in\mathbb{Z}^{d}}\widehat{\mu}\left( k,\xi\right)
\frac{e^{ik\cdot x}}{\left(  2\pi\right)  ^{d/2}}.
\]

\begin{lemma}
Let $\mu\in\mathcal{M}_{+}\left(  T^{\ast}\mathbb{T}^{d}\right)  $
and
$\Lambda\in\mathcal{L}$. The distribution:%
\[
\left\langle \mu\right\rangle _{\Lambda}\left(  x,\xi\right)
:=\sum _{k\in\Lambda}\widehat{\mu}\left(  k,\xi\right)
\frac{e^{ik\cdot x}}{\left(
2\pi\right)  ^{d/2}}%
\]
is a finite, positive Radon measure on $T^{\ast}\mathbb{T}^{d}$.
\end{lemma}

\begin{proof}
Let $a\in C_{c}^{\infty}\left(  T^{\ast}\mathbb{T}^{d}\right)  $
and $\left\{ v_{1},...,v_{n}\right\}  $ be a basis of
$\Lambda^{\perp}$. Suppose
\[
a\left(  x,\xi\right)  =\sum_{k\in\mathbb{Z}^{d}}\widehat{a}\left(
k,\xi\right)  \frac{e^{ik\cdot x}}{\left(  2\pi\right)  ^{d/2}};
\]
then it is not difficult to see that%
\begin{align*}
\left\langle a\right\rangle _{\Lambda}\left(  x,\xi\right)   &
:=\lim
_{T_{1},...,T_{n}\rightarrow\infty}\frac{1}{T_{1}...T_{n}}\int_{0}^{T_{1}%
}...\int_{0}^{T_{n}}a\left(  x+\sum_{j=1}^{n}t_{j}v_{j},\xi\right)
dt_{1}...dt_{n}\\
&  =\sum_{k\in\Lambda}\widehat{a}\left(  k,\xi\right)  \frac{e^{ik\cdot x}%
}{\left(  2\pi\right)  ^{d/2}},
\end{align*}
that $\left\langle a\right\rangle _{\Lambda}$ is non-negative as
soon as $a$ is, $\left\Vert \left\langle a\right\rangle
_{\Lambda}\right\Vert _{L^{\infty }\left(
T^{\ast}\mathbb{T}^{d}\right)  }\leq\left\Vert a\right\Vert
_{L^{\infty}\left(  T^{\ast}\mathbb{T}^{d}\right)  }$, and that
$\left\langle
a\right\rangle _{\Lambda}\in C_{c}^{\infty}\left(  T^{\ast}\mathbb{T}%
^{d}\right)  $ as well. Therefore,%
\[
\left\langle \left\langle \mu\right\rangle
_{\Lambda},a\right\rangle
=\int_{T^{\ast}\mathbb{T}^{d}}\left\langle a\right\rangle
_{\Lambda}\left( x,\xi\right)  \mu\left(  dx,d\xi\right)
\]
defines a positive distribution, which is a positive Radon measure
by Schwartz's theorem.
\end{proof}

Recall that a measure $\mu\in\mathcal{M}_{+}\left(  T^{\ast}\mathbb{T}%
^{d}\right)  $ is invariant under the action of the geodesic
flow\footnote{In what follows, we shall refer to such a measure
simply as a \emph{positive invariant measure}.} on
$T^{\ast}\mathbb{T}^{d}$ whenever:
\begin{equation}
\left(  \phi_{\tau}\right)  _{\ast}\mu=\mu,\quad\text{with
}\phi_{\tau}\left(
x,\xi\right)  =\left(  x+\tau\xi,\xi\right)  , \label{inv}%
\end{equation}
for all $\tau\in\R$.
Let us also introduce, for $v\in\mathbb{R}^{d}$ the translations
$\tau^{v}:T^{\ast}\mathbb{T}^{d}\rightarrow T^{\ast
}\mathbb{T}^{d}$ defined by:%
\[
\tau^{v}\left(  x,\xi\right)  =\left(  x+v,\xi\right)  .
\]

\begin{lemma}
\label{LemmaFourierInvariant}Let $\mu$ be a positive invariant
measure on $T^{\ast}\mathbb{T}^{d}$. Then every term in the
decomposition (\ref{dec}) is a positive invariant measure, and
\begin{equation}
\mu\rceil_{\mathbb{T}^{d}\times R_{\Lambda}}=\left\langle
\mu\right\rangle
_{\Lambda}\rceil_{\mathbb{T}^{d}\times R_{\Lambda}}. \label{m=mav}%
\end{equation}
Moreover, this last identity is equivalent to the following
invariance
property:%
\[
  \tau^{v}
_{\ast}\mu\rceil_{\mathbb{T}^{d}\times
R_{\Lambda}}=\mu\rceil_{\mathbb{T}^{d}\times
R_{\Lambda}},\quad\text{for every } v\in\Lambda^{\perp}.
\]

\end{lemma}

\begin{proof}
The invariance of the measures $\mu\rceil_{\mathbb{T}^{d}\times
R_{\Lambda}}$ is clearly a consequence of that of $\mu$ and of the
form of the geodesic flow on $T^{\ast}\mathbb{T}^{d}$. To check
(\ref{m=mav}) is suffices to show that $\widehat{\mu}\left(
k,\cdot\right)  \rceil_{R_{\Lambda}}=0$ as soon as
$k\notin\Lambda$. Start noticing that (\ref{inv}) is equivalent to
the fact that $\mu$ solves the equation:
\[
\xi\cdot\nabla_{x}\mu\left(  x,\xi\right)  =0.
\]
This is in turn equivalent to:%
\[
i\left(  k\cdot\xi\right)  \widehat{\mu}\left(  k,\xi\right)
=0,\qquad
\text{for every }k\in\mathbb{Z}^{d}\text{,}%
\]
from which we infer:
\begin{equation}
\operatorname*{supp}\widehat{\mu}\left(  k,\cdot\right)
\subset\left\{
\xi\in\mathbb{R}^{d}:k\cdot\xi=0\right\}  . \label{suppmuk}%
\end{equation}
Now remark that $R_{\Lambda}\cap\left\{
\xi\in\mathbb{R}^{d}:k\cdot \xi=0\right\}  \neq\emptyset$ if and
only if $k\in\Lambda$. This concludes the proof of the lemma.
\end{proof}

\section{Second microlocalization on a resonant affine subspace
\label{s:second}}
We now start with our main construction. Theorem \ref{t:main} (i) and Corollary \ref{t:example} will be proved at the end of \S \ref{s:successive}, and Theorem \ref{t:precise} in \S \ref{s:propagation}.

Given $\Lambda\in\mathcal{L}$, we denote by $\mathcal{S}_{\Lambda}^{1}$ the
class of smooth functions $a\left(  x,\xi,\eta\right)  $ on $T^{\ast
}\mathbb{T}^{d}\times\la\Lambda\ra$ that are:

\begin{enumerate}
\item[(i)] compactly supported w.r.t. $\left(  x,\xi\right)  \in T^{\ast
}\mathbb{T}^{d}$,

\item[(ii)] homogeneous of degree zero at infinity in $\eta\in\la\Lambda\ra$.
That is, if we denote by $\mathbb{S}_{\la\Lambda\ra}$ the unit sphere in
${\la\Lambda\ra}$ (\emph{i.e.} $\mathbb{S}_{\la\Lambda\ra}:=\left\langle
\Lambda\right\rangle \cap\mathbb{S}^{d-1}$) there exist $R_{0}>0$ and
$a_{\text{hom}}\in C_{c}^{\infty}\left(  T^{\ast}\mathbb{T}^{d}\times
\mathbb{S}_{\la\Lambda\ra}\right)  $ with%
\[
a\left(  x,\xi,\eta\right)  =a_{\text{hom}}\left(  x,\xi,\frac{\eta
}{\left\vert \eta\right\vert }\right)  \text{,\quad for }\left\vert
\eta\right\vert >R_{0}\text{ and }\left(  x,\xi\right)  \in T^{\ast}%
\mathbb{T}^{d}\text{;}%
\]
we also write
\[
a\left(  x,\xi,\infty\eta\right)  =a_{\text{hom}}\left(  x,\xi,\frac{\eta
}{\left\vert \eta\right\vert }\right)  \text{,\quad for } \eta\not= 0\text{;}%
\]

\item[(iii)] such that their non-vanishing Fourier coefficients (in the $x$
variable) correspond to frequencies $k\in\Lambda$:%
\[
a\left(  x,\xi,\eta\right)  =\sum_{k\in\Lambda}\widehat{a}\left(  k,\xi
,\eta\right)  \frac{e^{ik\cdot x}}{\left(  2\pi\right)  ^{d/2}}.
\]
We will also express this fact by saying that $a$\emph{ has only }%
$x$\emph{-Fourier modes in $\Lambda$.}
\end{enumerate}

Let $\left(  u_{h}\right)  $ be a bounded sequence in $L^{2}\left(
\mathbb{T}^{d}\right)  $ and suppose that its Wigner distributions
$w_{h}\left(  t\right)  :=w_{U_{V}\left(  t\right)  u_{h}}^{h}$ converge to a
semiclassical measure $\mu\in L^{\infty}\left(  \mathbb{R};\mathcal{M}%
_{+}\left(  T^{\ast}\mathbb{T}^{d}\right)  \right)  $ in the weak-$\ast$
topology of $L^{\infty}\left(  \mathbb{R};\mathcal{D}^{\prime}\left(  T^{\ast
}\mathbb{T}^{d}\right)  \right)  $.

Our purpose in this section is to analyse the structure of the restriction
$\mu\rceil_{\mathbb{T}^{d}\times R_{\Lambda}}$. To achieve this we shall
introduce a two-microlocal distribution describing the concentration of the
sequence $\left(  U_{V}\left(  t\right)  u_{h}\right)  $ on the resonant
subspaces:%
\[
\Lambda^{\perp}=\left\{  \xi\in\mathbb{R}^{d}:P_{\Lambda}\left(  \xi\right)
=0\right\}  ,
\]
where $P_{\Lambda}$ denotes the orthogonal projection of $\mathbb{R}^{d}$ onto
$\left\langle \Lambda\right\rangle $. Similar objects have been introduced in
the local, Euclidean, case by Nier \cite{NierScat} and Fermanian-Kammerer
\cite{FermanianShocks, Fermanian2micro} under the name of two-microlocal
semiclassical measures. A specific concentration scale may also be specified
in the two-microlocal variable, giving rise to the two-scale semiclassical
measures studied by Miller \cite{MillerThesis} and G\'{e}rard and
Fermanian-Kammerer \cite{FermanianGerardCroisements}. We shall follow the
approach in \cite{Fermanian2micro}, although it will be important to take into
account the global nature of the objects we shall be dealing with.

By Lemma \ref{LemmaFourierInvariant}, it suffices to characterize
the action of $\mu\rceil_{\mathbb{T}^{d}\times R_{\Lambda}}$ on test functions
having only $x$-Fourier modes in $\Lambda$. With this in mind, we introduce
two auxiliary distributions which describe more precisely how $w_{h}\left(
t\right)  $ concentrates along $\mathbb{T}^{d}\times\Lambda^{\perp}$ and that
act on symbols on the class $\mathcal{S}_{\Lambda}^{1}$.

Let $\chi\in C_{c}^{\infty}\left(  \mathbb{R}\right)  $ be a nonnegative
cut-off function that is identically equal to one near the origin. Let $R>0$. For
$a\in\cS_{\Lambda}^{1}$, we define%
\[
\left\langle w_{h,R}^{\Lambda}\left(  t\right)  ,a\right\rangle :=\int
_{T^{\ast}\mathbb{T}^{d}}\left(  1-\chi\left(  \frac{P_{\Lambda}\left(
\xi\right)  }{Rh}\right)  \right)  a\left(  x,\xi,\frac{P_{\Lambda}\left(
\xi\right)  }{h}\right)  w_{h}\left(  t\right)  \left(  dx,d\xi\right)  ,
\]
and%
\begin{equation}
\left\langle w_{\Lambda,h,R}\left(  t\right)  ,a\right\rangle :=\int_{T^{\ast
}\mathbb{T}^{d}}\chi\left(  \frac{P_{\Lambda}\left(  \xi\right)  }{Rh}\right)
a\left(  x,\xi,\frac{P_{\Lambda}\left(  \xi\right)  }{h}\right)  w_{h}\left(
t\right)  \left(  dx,d\xi\right)  . \label{subL}%
\end{equation}

\begin{remark}
If $\Lambda=\left\{  0\right\}  $ then $w_{h,R}^{\Lambda}=0$ and
$w_{\Lambda,h,R}\left(  t\right)  =w_{h}\left(  t\right)  \otimes\delta_{0}$.
\end{remark}

\begin{remark}
\label{Rmk Sum}For every $R>0$ and $a\in\mathcal{S}_{\Lambda}^{1}$ the
following holds.
\[
\int_{T^{\ast}\mathbb{T}^{d}}a\left(  x,\xi,\frac{P_{\Lambda}\left(
\xi\right)  }{h}\right)  w_{h}\left(  t\right)  \left(  dx,d\xi\right)
=\left\langle w_{h,R}^{\Lambda}\left(  t\right)  ,a\right\rangle +\left\langle
w_{\Lambda,h,R}\left(  t\right)  ,a\right\rangle .
\]

\end{remark}

The Calder\'{o}n-Vaillancourt theorem (see the appendix for a precise
statement) ensures that both $w_{h,R}^{\Lambda}$ and $w_{\Lambda,h,R}$ are
bounded in $L^{\infty}\left(  \mathbb{R};\left(  \mathcal{S}_{\Lambda}%
^{1}\right)  ^{\prime}\right)  $. After possibly extracting subsequences, we
have the existence of a limit~: for every $\varphi\in L^{1}\left(
\mathbb{R}\right)  $ and $a\in\mathcal{S}_{\Lambda}^{1}$,%
\[
\int_{\mathbb{R}}\varphi\left(  t\right)  \left\langle \tilde{\mu}^{\Lambda
}\left(  t,\cdot\right)  ,a\right\rangle dt:=\lim_{R\rightarrow\infty}%
\lim_{h\rightarrow0^{+}}\int_{\mathbb{R}}\varphi\left(  t\right)  \left\langle
w_{h,R}^{\Lambda}\left(  t\right)  ,a\right\rangle dt,
\]
and%
\begin{equation}
\int_{\mathbb{R}}\varphi\left(  t\right)  \left\langle \tilde{\mu}_{\Lambda
}\left(  t,\cdot\right)  ,a\right\rangle dt:=\lim_{R\rightarrow\infty}%
\lim_{h\rightarrow0^{+}}\int_{\mathbb{R}}\varphi\left(  t\right)  \left\langle
w_{\Lambda,h,R}\left(  t\right)  ,a\right\rangle dt. \label{doublelim}%
\end{equation}

Define, for $\left(  x,\xi,\eta\right)  \in T^{\ast}\mathbb{T}^{d}%
\times\mathbb{R}^{d}$ and $\tau\in\mathbb{R}$,%
\[
\phi_{\tau}^{0}\left(  x,\xi,\eta\right)  :=\left(  x+\tau\xi,\xi,\eta\right)
,
\]
and, when $\eta\neq0$,%
\[
\phi_{\tau}^{1}\left(  x,\xi,\eta\right)  :=\left(  x+\tau\frac{\eta
}{\left\vert \eta\right\vert },\xi,\eta\right)  .
\]
Since the distributions\footnote{It is convenient to use the word
``distribution'', but we actually mean elements of $L^{\infty}\left(
\mathbb{R};\left(  \mathcal{S}_{\Lambda}^{1}\right)  ^{\prime}\right)  $.}
$w_{h,R}^{\Lambda}$ and $w_{\Lambda,h,R}$ satisfy a transport equation with
respect to the $\xi$-variable the following result holds.

\begin{lemma}
\label{Lemma Inv}The distributions $\tilde{\mu}_{\Lambda}\left(
t,\cdot\right)  $ and $\tilde{\mu}^{\Lambda}\left(  t,\cdot\right)  $ are
$\phi_{\tau}^{0}$-invariant for almost every $t$:%
\[
\left(  \phi_{\tau}^{0}\right)  _{\ast}\tilde{\mu}_{\Lambda}\left(
t,\cdot\right)  =\tilde{\mu}_{\Lambda}\left(  t,\cdot\right)  ,\quad\left(
\phi_{\tau}^{0}\right)  _{\ast}\tilde{\mu}^{\Lambda}\left(  t,\cdot\right)
=\tilde{\mu}^{\Lambda}\left(  t,\cdot\right)  ,\quad\text{for every }\tau
\in\mathbb{R}%
\]

\end{lemma}

\begin{proof}
Let $a\in C_{c}^{\infty}\left(  T^{\ast}\mathbb{T}^{d}\right)  $. Then%
\begin{equation}
\frac{d}{dt}\left\langle w_{h}\left(  t\right)  ,a\right\rangle =i\left\langle
u_{h}\left(  t,\cdot\right)  ,\left[  -\frac{1}{2}\Delta+V\left(
t,\cdot\right)  ,\operatorname*{Op}\nolimits_{h}\left(  a\right)  \right]
u_{h}\left(  t,\cdot\right)  \right\rangle . \label{e:comm}%
\end{equation}
Now, using identity (\ref{e:weylcomm}) for the Weyl quantization
we deduce:
\begin{equation}
\frac{d}{dt}\left\langle w_{h}\left(  t\right)  ,a\right\rangle =\frac{1}%
{h}\left\langle w_{h}\left(  t\right)  ,\xi\cdot\partial_{x}a\right\rangle
+\left\langle \mathcal{L}_{V}^{h}\left(  t\right)  ,a\right\rangle ,
\label{e:propW}%
\end{equation}
where%
\begin{equation}
\left\langle \mathcal{L}_{V}^{h}\left(  t\right)  ,a\right\rangle
:=i\left\langle u_{h}\left(  t,\cdot\right)  ,\left[  V\left(  t,\cdot\right)
,\operatorname*{Op}\nolimits_{h}\left(  a\right)  \right]  u_{h}\left(
t,\cdot\right)  \right\rangle . \label{e:LV}%
\end{equation}
Note that this quantity is bounded in $h$ for $t$ varying on a compact set.
Integration in $t$ against a function $\varphi\in C_{c}^{1}\left(
\mathbb{R}\right)  $ gives:%
\[
\int_{\mathbb{R}}\varphi\left(  t\right)  \left\langle w_{h}\left(  t\right)
,\xi\cdot\partial_{x}a\right\rangle dt=-h\int_{\mathbb{R}}\varphi^{\prime
}\left(  t\right)  \left\langle w_{h}\left(  t\right)  ,a\right\rangle
dt-h\int_{\mathbb{R}}\varphi\left(  t\right)  \left\langle \mathcal{L}_{V}%
^{h}\left(  t\right)  ,a\right\rangle dt.
\]
Replacing $a$ in the above identity by%
\[
\chi\left(  \frac{P_{\Lambda}\left(  \xi\right)  }{Rh}\right)  a\left(
x,\xi,\frac{P_{\Lambda}\left(  \xi\right)  }{h}\right)  \mbox{ or} \left(
1-\chi\left(  \frac{P_{\Lambda}\left(  \xi\right)  }{Rh}\right)  \right)
a\left(  x,\xi,\frac{P_{\Lambda}\left(  \xi\right)  }{h}\right)
\]
and letting $h\rightarrow0^{+}$ and $R\rightarrow\infty$ we obtain:%
\[
\left\langle \tilde{\mu}_{\Lambda}\left(  t,\cdot\right)  ,\xi\cdot
\partial_{x}a\right\rangle =0 \mbox{ and }\left\langle \tilde{\mu}^{\Lambda
}\left(  t,\cdot\right)  ,\xi\cdot\partial_{x}a\right\rangle =0
\]
which is the desired invariance property.
\end{proof}

Positivity and invariance properties of the accumulation points $\tilde{\mu}_{\Lambda}\left(
t,\cdot\right)  $ and $\tilde{\mu}^{\Lambda}\left(  t,\cdot\right)  $ are
described in the next two results.

\begin{theorem}
\label{thm 1st2micro}(i)  For a.e. $t\in\mathbb{R}$, $\tilde{\mu}^{\Lambda}\left(  t,\cdot\right)
$ is positive, $0$-homogeneous and
supported at infinity in the variable $\eta$ ($i.e.$, it vanishes when paired
with a compactly supported function). As a consequence, $\tilde{\mu}^{\Lambda
}\left(  t,\cdot\right)  $ may be identified \footnote{\label{f:homogen}More precisely, there exists a positive measure $M^{\Lambda}\left(  t,\cdot\right)
$ on $T^{\ast}\mathbb{T}^{d}\times\mathbb{S}_{\la\Lambda\ra}$
such that $\int_{T^*\T^d\times\la \Lambda\ra} a(x, \xi, \eta)\tilde{\mu}^{\Lambda}\left(  t,d\xi, d\eta\right)=
\int_{T^*\T^d\times\mathbb{S}_{\la\Lambda\ra}} a(x, \xi,\infty \eta)M^{\Lambda}\left(  t,d\xi, d\eta\right).$
For simplicity we will identify $M^{\Lambda}\left(  t,\cdot\right)$ and $\tilde{\mu}^{\Lambda
}\left(  t,\cdot\right)  $, and we will write the integrals in the most convenient way according to the context.
} with a positive measure on
$T^{\ast}\mathbb{T}^{d}\times\mathbb{S}_{\la\Lambda\ra} $.

For a.e. $t\in\mathbb{R}$, the projection of $\tilde{\mu}_{\Lambda
}\left(  t,\cdot\right)  $
on $T^{\ast}\mathbb{T}^{d}$ is a positive measure.

(ii) Both   $\tilde{\mu}^{\Lambda}\left(  t,\cdot\right)
$ and $\tilde{\mu}_{\Lambda
}\left(  t,\cdot\right)  $ are $\phi_{\tau}^{0}$-invariant.

(iii) Let
\[
\mu^{\Lambda}\left(  t,\cdot\right)  :=\int_{\left\langle \Lambda\right\rangle
}\tilde{\mu}^{\Lambda}\left(  t,\cdot,d\eta\right)  \rceil_{(x, \xi)\in\mathbb{T}%
^{d}\times R_{\Lambda}},\quad\mu_{\Lambda}\left(  t,\cdot\right)
:=\int_{ {\la\Lambda\ra}}\tilde{\mu}_{\Lambda}\left(  t,\cdot
,d\eta\right)  \rceil_{(x, \xi)\in\mathbb{T}^{d}\times R_{\Lambda}}.
\]
Then both $\mu^{\Lambda}\left(  t,\cdot\right)  $ and $\mu_{\Lambda}\left(
t,\cdot\right)  $ are positive measures on $T^{\ast}\mathbb{T}^{d}$, invariant
by the geodesic flow, and satisfy:%
\begin{equation}
\mu\left(  t,\cdot\right)  \rceil_{\mathbb{T}^{d}\times R_{\Lambda}}%
=\mu^{\Lambda}\left(  t,\cdot\right)  +\mu_{\Lambda}\left(  t,\cdot\right)  .
\label{decomposition}%
\end{equation}

\end{theorem}

Note that identity (\ref{decomposition}) is a consequence of the decomposition
property expressed in Remark \ref{Rmk Sum}.

The following result is the key step of our proof, it states that
both $\mu^{\Lambda}$ and $\mu_{\Lambda}$ have some extra
regularity in the variable $x$, for two different reasons~:

\begin{theorem}
\label{Thm Properties}(i) For a.e. $t\in\mathbb{R}$,
$\tilde\mu_{\Lambda}\left(  t,\cdot\right)  $ is concentrated on $\mathbb{T}%
^{d}\times\Lambda^{\bot}\times\la\Lambda\ra$ and its projection on
$\mathbb{T}^{d}$ is absolutely continuous with respect to the Lebesgue measure.

(ii) For a.e. $t\in\mathbb{R}$, the measure $\tilde{\mu}^{\Lambda}\left(
t,\cdot\right)  $ satisfies the invariance property:%
\begin{equation}
\left(  \phi_{\tau}^{1}\right)  _{\ast}\tilde{\mu}^{\Lambda}\left(
t,\cdot\right)  =\tilde{\mu}^{\Lambda}\left(  t,\cdot\right)  ,\quad
\tau\in\mathbb{R}. \label{2minv}%
\end{equation}

\end{theorem}

\begin{remark}
As we shall prove in Section \ref{s:propagation}, the distributions $\tilde\mu_{\Lambda
}\left(  t,\cdot\right)  $ verify a propagation law that is related to unitary
propagator generated by the self-adjoint operator $\frac{1}{2}\Delta
+\left\langle V\right\rangle _{\Lambda}\left(  t,\cdot\right)  $, where
$\left\langle V\right\rangle _{\Lambda}$ denotes the average of $V$ along
$\Lambda^{\perp}$.
\end{remark}

\begin{remark}
\label{r:nice}The invariance property (\ref{2minv}) provides $\tilde{\mu
}^{\Lambda}$ with additional regularity. This is clearly seen when
$\operatorname*{rk}\Lambda=1$. In that case, (\ref{2minv}) implies that, for a.e. $t\in\mathbb{R}$, the measure
$\tilde{\mu}^{\Lambda}\left(  t,\cdot\right)  $ satisfies for every
$v\in\mathbb{S}_{\la\Lambda\ra} $:%
\begin{equation}
\left(  \tau_{s}^{v}\right)  _{\ast}\tilde{\mu}^{\Lambda}\left(
t,\cdot\right)  \rceil_{\mathbb{T}^{d}\times R_{\Lambda}\times
 {\la\Lambda\ra} }=\tilde{\mu}^{\Lambda}\left(
t,\cdot\right) \rceil_{\mathbb{T}^{d}\times
R_{\Lambda}\times {\la\Lambda\ra}
},\quad s\in\mathbb{R}\text{.} \label{invmusup}%
\end{equation}
On the other hand,  Lemma
\ref{LemmaFourierInvariant} implies that
(\ref{invmusup}) also holds for every $v\in\Lambda^{\perp}$. Therefore, we
conclude that $\tilde{\mu}^{\Lambda}\left(  t,\cdot\right)  \rceil
_{\mathbb{T}^{d}\times R_{\Lambda}\times{\la\Lambda\ra} }$ is
constant in $x\in\mathbb{T}^{d}$ in this case.
\end{remark}

\begin{remark}
Theorems \ref{thm 1st2micro} (iii), and \ref{Thm Properties} (i), together with
Lemma \ref{Lemma decomposition} imply that, for a.e. $t\in\mathbb{R}$, we have
a decomposition:%
\[
\mu\left(  t,\cdot\right)  =\sum_{\Lambda\in\mathcal{L}}\mu^{\Lambda}\left(
t,\cdot\right)  +\sum_{\Lambda\in\mathcal{L}}\mu_{\Lambda}\left(
t,\cdot\right)  ,
\]
where the second term in the above sum defines a positive measure whose
projection on $\mathbb{T}^{d}$ is absolutely continuous with respect to the
Lebesgue measure.
\end{remark}

The rest of this section is devoted to the proofs of Theorems
\ref{thm 1st2micro} and \ref{Thm Properties}.

\subsection{Computation and structure of $\tilde{\mu}_{\Lambda}$\label{s:sy}}

We use the linear isomorphism%
\[
\chi_{\Lambda}:\Lambda^{\perp}\times\left\langle
\Lambda\right\rangle \rightarrow\mathbb{R}^{d}:\left(  s,y\right)  \mapsto s+y
\]
and denote by
$\tilde{\chi}_{\Lambda}:T^{\ast}\Lambda^{\perp}\times T^{\ast
}\left\langle \Lambda\right\rangle \rightarrow
T^{\ast}\mathbb{R}^{d}$ the induced canonical transformation. Explicitly, $\tilde{\chi}_{\Lambda}$
goes as follows~: let $(s, \sigma)\in T^{\ast}\Lambda^{\perp}=\Lambda^{\perp}\times \left(\Lambda^{\perp}\right)^*$ and $(y, \eta)\in T^{\ast
}\left\langle \Lambda\right\rangle=\left\langle \Lambda\right\rangle\times \left\langle \Lambda\right\rangle^*$. Extend $\sigma$ to a linear form on $\R^d$ vanishing on $\left\langle \Lambda\right\rangle$, and $\eta$ to a linear form on $\R^d$ vanishing on $\Lambda^{\perp}$.
Then $\tilde{\chi}_{\Lambda}(s, \sigma, y, \eta)=(s+y, \sigma+\eta)\in T^*\R^d=\R^d\times (\R^d)^*$.

The map $\chi_{\Lambda}$ goes to the quotient and gives a smooth Riemannian covering~:
\[
\pi_{\Lambda}:\mathbb{T}_{\Lambda^{\perp}}\times\mathbb{T}_{\Lambda
}\rightarrow\mathbb{T}^{d}:\left(  s,y\right)  \mapsto s+y;
\] $\tilde{\pi}_{\Lambda}$ will denote its
extension to the cotangent bundles $T^{\ast}\mathbb{T}_{\Lambda^{\perp}}\times
T^{\ast}\mathbb{T}_{\Lambda}\rightarrow T^{\ast}\mathbb{T}^{d}$. Let
$p_\Lambda$ denote the degree of $\pi_\Lambda$. 

There is a linear
isomorphism $T_{\Lambda
}:L_{\text{loc}}^{2}\left(  \mathbb{R}^{d}\right)  \rightarrow L_{\text{loc}%
}^{2}\left(  \Lambda^{\perp}\times\left\langle
\Lambda\right\rangle \right)  $ given by
\[
T_{\Lambda}u:=\frac{1}{\sqrt{p_{\Lambda}}}\left(u\circ\chi_{\Lambda}\right).
\]
Note that because of the factor $p_{\Lambda}^{-1/2}$,
$T_{\Lambda}$ maps
$L^{2}\left(  \mathbb{T}^{d}\right)  $ isometrically into a subspace of $L^{2}\left(  \mathbb{T}%
_{\Lambda^{\perp}}\times\mathbb{T}_{\Lambda}\right)  =L^{2}\left(
\mathbb{T}_{\Lambda^{\perp}};L^{2}\left(  \mathbb{T}_{\Lambda}\right)
\right)  $. Moreover, $T_{\Lambda}$ maps $L_{\Lambda}%
^{2}\left(  \mathbb{T}^{d}\right)  $ into $L^{2}\left(  \mathbb{T}_{\Lambda
}\right)  \subset L^{2}\left(  \mathbb{T}_{\Lambda^{\perp}}\times
\mathbb{T}_{\Lambda}\right)  $, since if the non-vanishing Fourier modes of
$u$ correspond only to frequencies $k\in\Lambda$, then
\begin{equation}
T_{\Lambda}u\left(  s,y\right)  =\frac{1}{\sqrt{p_{\Lambda}}}u\left(  y\right)  \quad\text{for every }%
s\in\mathbb{T}_{\Lambda^{\perp}}. \label{constantins}%
\end{equation}
Since $\tilde{\chi}_{\Lambda}$ is linear, the following holds for any $a\in
C^{\infty}\left(  T^{\ast}\mathbb{R}^{d}\right)  $:%
\[
T_{\Lambda}\Op_{h}\left(  a\right)  =\Op_{h}\left(  a\circ\tilde{\chi
}_{\Lambda}\right)  T_{\Lambda}.
\]
Denote by $\Op_{h}^{\Lambda^{\perp}}$ and $\Op_{h}^{\Lambda}$ the Weyl
quantization operators defined on smooth test functions on $T^{\ast}%
\Lambda^{\perp}\times T^{\ast}\left\langle \Lambda\right\rangle $ which act on
the variables $T^{\ast}\Lambda^{\perp}$ and $T^{\ast}\left\langle
\Lambda\right\rangle $ respectively, leaving the other frozen. The composition
$\Op_{h}^{\Lambda^{\perp}}\Op_{h}^{\Lambda}$ gives the whole Weyl quantization
$\Op_{h}$ on $T^{\ast}\Lambda^{\perp}\times T^{\ast}\left\langle
\Lambda\right\rangle $. Now, if $a\in\mathcal{S}_{\Lambda}^{1}$ we have, in
view of (\ref{constantins}), that $a\circ\tilde{\pi}_{\Lambda}$ does not
depend on $s\in\mathbb{T}_{\Lambda^{\perp}}$ and therefore we write $a\circ\tilde{\pi}_{\Lambda}(\sigma, y, \eta)$ for $a\circ\tilde{\pi}_{\Lambda}(s, \sigma, y, \eta)$. We have
\begin{equation}
T_{\Lambda}\Op_{h}\left(  a\right)  =\Op_{h}^{\Lambda}\left(  a\circ\tilde
{\pi}_{\Lambda}\left(  hD_{s},\cdot\right)  \right)  T_{\Lambda}.
\label{ChangeCoord}%
\end{equation}
Note that for every $\sigma\in\Lambda^{\perp}$, the operators $\Op_{h}%
^{\Lambda}\left(  a\circ\tilde{\pi}_{\Lambda}\left(  \sigma,\cdot\right)
\right)  $ map $L^{2}\left(  \mathbb{T}_{\Lambda}\right)  $ into itself. To be even more precise, it maps  the subspace
$T_\Lambda(L^2_\Lambda(\T^d))$ into itself.

\begin{remark}
\label{rmkCompact} Let $a\in\mathcal{S}_{\Lambda}^{1}$; set $a_{R}\left(
x,\xi,\eta\right)  :=\chi\left(  \eta/R\right)  a\left(  x,\xi,\eta\right)  $
and define $a_{R,\Lambda}^{h}\in C_{c}^{\infty}\left(  \Lambda^{\perp}\times
T^{\ast}\mathbb{T}_{\Lambda}\right)  $ by%
\[
a_{R,\Lambda}^{h}\left(  \sigma,y,\eta\right)  :=a_{R}\left(  \tilde{\pi
}_{\Lambda}\left(  \sigma,y,h\eta\right)  ,\eta\right)  =a_R(y, \sigma+h\eta, \eta),\quad\left(
y,\eta\right)  \in T^{\ast}\mathbb{T}_{\Lambda},\quad\sigma\in\Lambda^{\perp
}.
\]
It is simple to check that (\ref{ChangeCoord}) gives:
\[
T_{\Lambda}\Op_{h}\left(  a\right)
T_{\Lambda}^{\ast}=\Op_{1}^{\Lambda }\left(
a_{R,\Lambda}^{h}\left(  hD_{s},\cdot\right)  \right)           ,
\]
and
\[
\left\langle w_{\Lambda,h,R}\left(  t\right)  ,a\right\rangle
=\left\langle T_{\Lambda}u_{h}\left(  t,\cdot\right)  ,
\Op_{1}^{\Lambda}\left(  a_{R,\Lambda}^{h}\left(
hD_{s},\cdot\right) \right)  T_{\Lambda}u_{h}\left( t,\cdot\right)
\right\rangle _{L^{2}\left(
\mathbb{T}_{\Lambda^{\perp}};L^{2}\left(
\mathbb{T}_{\Lambda}\right) \right)  }.
\]

Note that for every $R>0$, $t\in\mathbb{R}$ and $\left(  s,\sigma\right)  \in
T^{\ast}\mathbb{T}_{\Lambda^{\perp}}$, the operator
\[
\Op_{1}^{\Lambda}\left(  a_{R,\Lambda}^{h}\left(  \sigma,\cdot\right)
\right)
\]
is compact on $L^{2}\left(  \mathbb{T}_{\Lambda}\right)  $, since
$a_{R,\Lambda}^{h}$ is compactly supported in the variable $\eta$.
\end{remark}

Given a Hilbert space $H$, denote respectively by $\mathcal{K}\left(
H\right)  $ and $\mathcal{L}^{1}\left(  H\right)  $ the spaces of compact and
trace class operators on $H$. A measure on a polish space $T$, taking values
in $\mathcal{L}^{1}\left(  H\right)  $, is defined as a bounded linear
functional $\rho$ from $C_{c}\left(  T\right)  $ to $\mathcal{L}^{1}\left(
H\right)  $; $\rho$ is said to be positive if, for every nonnegative $b\in
C_{c}\left(  T\right)  $, $\rho\left(  b\right)  $ is a positive hermitian
operator. The set of such measures is denoted by $\mathcal{M}_{+}\left(
T;\mathcal{L}^{1}\left(  H\right)  \right)  $; they can be identified in a
natural way to positive linear functionals on $C_{c}\left(  T;\mathcal{K}%
\left(  H\right)  \right)  $. Background and further details on
operator-valued measures may be found for instance in \cite{GerardMDM91}.

In view of Remark \ref{rmkCompact}, it turns out that the limiting object
relevant in the computation of $\tilde{\mu}_{\Lambda}$ is the one presented in
the next result. For $K\in C_{c}^{\infty}\left(  T^{\ast}\mathbb{T}%
_{\Lambda^{\perp}};\mathcal{K}\left(  L^{2}\left(  \mathbb{T}_{\Lambda
}\right)  \right)  \right)  $ denote:%
\begin{equation}
\label{e:nh}\left\langle n_{h}^{\Lambda}\left(  t\right)
,K\right\rangle :=\left\langle T_{\Lambda}U_{V}\left(  t\right)
u_{h},K\left(  s,hD_{s}\right)  T_{\Lambda}U_{V}\left(  t\right)
u_{h}\right\rangle _{L^{2}\left(
\mathbb{T}_{\Lambda^{\perp}};L^{2}\left(
\mathbb{T}_{\Lambda}\right)  \right)  }.
\end{equation}

\begin{proposition}
\label{p:weakstarlimit} Suppose $\left(  u_{h}\right)  $ is bounded in
$L^{2}\left(  \mathbb{T}^{d}\right)  $. Then, modulo a subsequence, the
following convergence takes place:
\begin{equation}
\label{e:rholambda}\lim_{h\rightarrow0^{+}}\int_{\mathbb{R}}\varphi\left(
t\right)  \left\langle n_{h}^{\Lambda}\left(  t\right)  ,K\right\rangle
dt=\int_{\mathbb{R}}\varphi\left(  t\right)  \Tr\int_{T^{\ast}\mathbb{T}%
_{\Lambda^{\perp}}}K\left(  s,\sigma\right)  \tilde{\rho}_{\Lambda}\left(
t,ds,d\sigma\right)  dt,
\end{equation}
for every $K\in C_{c}^{\infty}\left(  T^{\ast}\mathbb{T}_{\Lambda^{\perp}%
};\mathcal{K}\left(  L^{2}\left(  \mathbb{T}_{\Lambda}\right)  \right)
\right)  $ and every $\varphi\in L^{1}\left(  \mathbb{R}\right)  $; in other
words, $\tilde{\rho}_{\Lambda}$ is the limit of $n_{h}^{\Lambda}\left(
t\right)  $ in the weak-$\ast$ topology of $L^{\infty}\left(  \R, \cD^{\prime
}\left(  T^{\ast}\mathbb{T}_{\Lambda^{\perp}};\mathcal{L}^{1}\left(
L^{2}\left(  \mathbb{T}_{\Lambda}\right)  \right)  \right)  \right)  $.

Then $\tilde{\rho}_{\Lambda}$ is an $L^{\infty}$-function in $t$ taking values
in the set of positive, $\mathcal{L}^{1}\left(  L^{2}\left(  \mathbb{T}%
_{\Lambda}\right)  \right)  $-valued measures on $T^{\ast}\mathbb{T}%
_{\Lambda^{\perp}}$. We have $\int_{T^{\ast}\mathbb{T}%
_{\Lambda^{\perp}}}\Tr\, \tilde{\rho}_{\Lambda}(t, ds, d\sigma)\leq 1$ for {\em a.e.} $t$.

Moreover, for almost every $t$ the measure $\tilde{\rho}_{\Lambda}\left(
t,\cdot\right)  $ is invariant by the geodesic flow $\phi_{\tau}|_{T^{\ast
}\mathbb{T}_{\Lambda^{\perp}}} : (s, \sigma)\mapsto(s+\tau\sigma, \sigma)$
($\tau\in\R$).
\end{proposition}

This result is the analogue of Theorems 1 and 2 of \cite{MaciaAv} in the
context of operator-valued measures. Its proof follows the lines of those
results, after the adaptation of the symbolic calculus to operator valued
symbols as developed for instance in \cite{GerardMDM91}.

When taking the limits $h\To 0$ and $R\To +\infty$ one should have in mind the
following facts. For any $a\in\mathcal{S}_{\Lambda}^{1}$, we have for fixed
$R$
\[
\Op_{1}^{\Lambda}\left(  a_{R,\Lambda}^{h}\left(  \sigma,\cdot\right)
\right)  =\Op_{1}^{\Lambda}\left(  a_{R,\Lambda}^{0}\left(  \sigma
,\cdot\right)  \right)  +\cO(h)
\]
where the remainder $\cO(h)$ is estimated in the operator norm (using the
Calder\'{o}n-Vaillancourt theorem). In addition, the following limit takes
place in the strong topology of $C_{c}^{\infty}\left(  T^{\ast}\mathbb{T}%
_{\Lambda^{\perp}};\mathcal{L}\left(  L^{2}\left(  \mathbb{T}_{\Lambda
}\right)  \right)  \right)  $\/:%
\[
\lim_{R\rightarrow\infty} \Op_{1}^{\Lambda}\left(  a_{R,\Lambda}^{0}\left(
\sigma,\cdot\right)  \right)  =\Op_{1}^{\Lambda}\left(  a_{\Lambda}^{0}\left(
\sigma,\cdot\right)  \right)  ,
\]
where $a_{\Lambda}^{0}$ is defined by setting $h=0$ and $R=\infty$ in the
definition of $a_{R,\Lambda}^{h}$. In other words, $a_{\Lambda}^{0}(\sigma, y,
\eta)=a(\tilde\pi_{\Lambda}(\sigma, y, 0), \eta)=a(y, \sigma, \eta)$.

Combining what we have done so far, we find

\begin{corollary}
Let $\tilde{\rho}_{\Lambda}\in L^{\infty}\left(  \mathbb{R};\mathcal{M}%
_{+}\left(  T^{\ast}\mathbb{T}_{\Lambda^{\perp}};\mathcal{L}^{1}\left(
L^{2}\left(  \mathbb{T}_{\Lambda}\right)  \right)  \right)  \right)  $ be a
weak-$\ast$ limit of $\left(  n_{h}^{\Lambda}\right)  $. Let $\tilde{\mu
}_{\Lambda}$ be defined by \eqref{subL} and \eqref{doublelim}. Then, for every
$a\in\mathcal{S}_{\Lambda}^{1}$ and a.e. $t\in\mathbb{R}$ we have:%
\begin{align*}
&  \int_{T^*\mathbb{T}^{d} \times\left\langle \Lambda
\right\rangle }a\left(  x,\xi,\eta\right)  \tilde{\mu}_{\Lambda}\left(
t,dx,d\xi,d\eta\right) \\
&  = \Tr\left(  \int_{T^{\ast}\mathbb{T}_{\Lambda^{\perp}}} \Op_{1}^{\Lambda
}\left(  a_{\Lambda}^{0}\left(  \sigma,\cdot\right)  \right)  \tilde{\rho
}_{\Lambda}\left(  t,ds,d\sigma\right)  \right)  .
\end{align*}

\end{corollary}

\begin{remark}
If $a\in\mathcal{S}_{\Lambda}^{1}$ does not depend on $\eta\in\mathbb{R}^{d}$
then the above identity can be rewritten as:%
\begin{equation}
\int_{T^*\mathbb{T}^{d} \times\left\langle \Lambda
\right\rangle }a\left(  x,\xi\right)  \tilde{\mu}_{\Lambda}\left(
t,dx,d\xi,d\eta\right)  =\Tr_{L^{2}\left(  \mathbb{T}_{\Lambda}\right)}\left(  \int_{T^{\ast}\mathbb{T}_{\Lambda^{\perp}%
}}m_{a\circ\pi_{\Lambda}}\left(  \sigma\right)  \tilde{\rho}_{\Lambda}\left(
t,ds,d\sigma\right)  \right)  , \label{e:muLtilda}%
\end{equation}
where for $\sigma\in\Lambda^{\perp}$, $m_{a}\left(  \sigma\right)  $ denotes
the operator of multiplication by $a\left( ., \sigma \right)  $ in $L^{2}\left(  \mathbb{T}_{\Lambda}\right)  $.

Since all the arguments above actually hold with $L^2 (\T_\Lambda)$ replaced by the smaller space
$T_\Lambda(L^2_\Lambda (\T^d))$, and since
$m_{a\circ\pi_{\Lambda}}(\sigma)%
=T_{\Lambda}m_{a}(\sigma)T_{\Lambda}^{\ast}$ on this space (where $m_{a}(\sigma)$ is again the multiplication operator by $a(., \sigma)$), we can
write the above identity as:%
\begin{equation}\label{e:intrinsic}
\int_{T^*\mathbb{T}^{d} \times\left\langle \Lambda
\right\rangle }a\left(  x, \xi\right)  \tilde{\mu}_{\Lambda}\left(  t,dx,d\xi
,d\eta\right)  =\Tr_{L^2_\Lambda (\T^d)}\left(  \int_{T^{\ast}\mathbb{T}_{\Lambda^{\perp}}}%
m_{a}(\sigma)T_{\Lambda}^{\ast}\tilde{\rho}_{\Lambda}\left(  t,ds,d\sigma\right)
T_{\Lambda}\right)  .
\end{equation}

And when $a=a(x)$ does not depend on $\xi$, this reduces to
\begin{equation} 
\int_{T^*\mathbb{T}^{d} \times\left\langle \Lambda
\right\rangle }a\left(  x\right)  \tilde{\mu}_{\Lambda}\left(  t,dx,d\xi
,d\eta\right)  =\Tr_{L^2_\Lambda (\T^d)}\left(  \int_{T^{\ast}\mathbb{T}_{\Lambda^{\perp}}}%
m_{a}T_{\Lambda}^{\ast}\tilde{\rho}_{\Lambda}\left(  t,ds,d\sigma\right)
T_{\Lambda}\right) 
\end{equation}
which proves the absolute continuity of the projection of $\tilde\mu_\Lambda$ to $\T^d$.
\end{remark}

\subsection{Computation and structure of $\tilde{\mu}^{\Lambda}$}

The positivity of $\tilde{\mu}^{\Lambda}\left(  t,\cdot\right)  $ can be
deduced following the lines of \cite{FermanianGerardCroisements} \S 2.1, or
those of the proof of Theorem 1 in \cite{GerardMDM91}; the idea is recalled in
Corollary \ref{CoroCV} in the appendix. Given $a\in\mathcal{S}_{\Lambda}^{1}$
there exists $R_{0}>0$ and $a_{\text{hom}}\in C_{c}^{\infty}\left(  T^{\ast
}\mathbb{T}^{d}\times\mathbb{S}_{\la\Lambda\ra}\right)  $ such that
\[
a\left(  x,\xi,\eta\right)  =a_{\text{hom}}\left(  x,\xi,\frac{\eta
}{\left\vert \eta\right\vert }\right)  ,\quad\text{for }\left\vert
\eta\right\vert \geq R_{0}.
\]
Clearly, for $R$ large enough, the value $\left\langle w_{h,R}^{\Lambda
}\left(  t\right)  ,a\right\rangle $ only depends on $a_{\text{hom}}$.
Therefore, the limiting distribution $\tilde{\mu}^{\Lambda}\left(
t,\cdot\right)  $ can be viewed as an element of the dual of $C_{c}^{\infty
}\left(  T^{\ast}\mathbb{T}^{d}\times\mathbb{S}_{\la\Lambda\ra}\right)  $. Let
us now check the invariance property (\ref{2minv}). Set
\[
a^{R}\left(  x,\xi,\eta\right)  :=\left(  1-\chi\left(  \frac{\eta}{R}\right)
\right)  a\left(  x,\xi,\eta\right)  .
\]
Notice that since $a$ has only Fourier modes in $\Lambda$:%
\[
\frac{\xi}{h}\cdot\partial_{x}a^{R}\left(  x,\xi,\frac{P_{\Lambda}\xi}%
{h}\right)  =\frac{P_{\Lambda}\xi}{h}\cdot\partial_{x}a^{R}\left(  x,\xi
,\frac{P_{\Lambda}\xi}{h}\right)  .
\]
Therefore, by equations (\ref{e:propW}) and (\ref{e:LV}), and taking into
account that $a^{R}$ vanishes near $\eta=0$, we have, for every $\varphi\in
C_{c}^{1}\left(  \mathbb{R}\right)  $:%
\begin{align}
\int_{\mathbb{R}}\varphi\left(  t\right)  \left\langle w_{h,R}^{\Lambda
}\left(  t\right)  ,\frac{\eta}{\left\vert \eta\right\vert }\cdot\partial
_{x}a^{R}\right\rangle dt  &  =-\int_{\mathbb{R}}\varphi^{\prime}\left(
t\right)  \left\langle w_{h,R}^{\Lambda}\left(  t\right)  ,\frac{1}{\left\vert
\eta\right\vert }a^{R}\right\rangle dt\label{e:wLinf}\\
&  +\int_{\mathbb{R}}\varphi\left(  t\right)  \left\langle \mathcal{L}_{V}%
^{h}\left(  t\right)  ,\frac{1}{\left\vert \eta\right\vert }a^{R}\right\rangle
dt. \label{e:wLinf2}%
\end{align}
Writing $\eta=r\omega$ with $r>0$ and $\omega\in\mathbb{S}_{\la\Lambda\ra}$ we
find, for $R$ large enough:%
\[
b^{R}\left(  x,\xi,\eta\right)  :=\frac{1}{\left\vert \eta\right\vert }%
a^{R}\left(  x,\xi,\eta\right)  =\frac{1}{r}\left(  1-\chi\left(  \frac{r}%
{R}\omega\right)  \right)  a_{\text{hom}}\left(  x,\xi,\omega\right)  ;
\]
moreover, since $b^{R}$ is homogeneous of degree $-1$ in the variable $\eta$,
the Calder\'{o}n-Vaillancourt theorem implies that the operator:%
\[
B_{h,R}^{\Lambda}:=\operatorname*{Op}\nolimits_{h}\left(  b^{R}\left(
x,\xi,\frac{P_{\Lambda}\xi}{h}\right)  \right)
\]
satisfies:%
\[
\limsup_{h\rightarrow0^{+}}\left\Vert B_{h,R}^{\Lambda}\right\Vert
_{\mathcal{L}\left(  L^{2}\left(  \mathbb{T}^{d}\right)  \right)  }\leq
\frac{C}{R}.
\]
Therefore,
\[
\lim_{R\rightarrow\infty}\lim_{h\rightarrow0^{+}}\int_{\mathbb{R}}%
\varphi^{\prime}\left(  t\right)  \left\langle w_{h,R}^{\Lambda}\left(
t\right)  ,\frac{1}{\left\vert \eta\right\vert }a^{R}\right\rangle dt=0,
\]
and%
\[
\limsup_{h\rightarrow0^{+}}\left\langle \mathcal{L}_{V}^{h}\left(  t\right)
,\frac{1}{\left\vert \eta\right\vert }a^{R}\right\rangle \leq C\limsup
_{h\rightarrow0^{+}}\left\Vert \left[  V,B_{h,R}^{\Lambda}\right]  \right\Vert
_{\mathcal{L}\left(  L^{2}\left(  \mathbb{T}^{d}\right)  \right)  }\leq
\frac{C^{\prime}}{R}\left\Vert V\right\Vert _{L^{\infty}\left(  \mathbb{T}%
^{d}\right)  }.
\]
After letting $h\rightarrow0^{+}$ and $R\rightarrow\infty$ in (\ref{e:wLinf}),
(\ref{e:wLinf2}), we conclude that for almost every $t\in\mathbb{R}$:%
\[
\omega\cdot\nabla_{x}\tilde{\mu}^{\Lambda}\left(  t,x,\xi,\omega\right)  =0.
\]
This is equivalent to (\ref{2minv}).

\section{Successive second microlocalizations corresponding to a sequence of
lattices\label{s:successive}}

Let us summarize what we have done in the previous section. The
semiclassical measure $\mu(t,.)$ has been decomposed as a sum
\[
\mu(t,.)=\sum_{\Lambda}\mu_{\Lambda}(t,.)+\sum_{\Lambda}\mu^{\Lambda}(t,.),
\]
where $\Lambda$ runs over the set of primitive submodules of
$\IZ^{d}$, and where
\[
\mu_{\Lambda}(t,.)=\int_{\la \Lambda\ra}\tilde{\mu}_{\Lambda}(t,.,d\eta
)\rceil_{\IT^{d}\times R_{\Lambda}},\qquad\mu^{\Lambda}(t,.)=\int_{\la \Lambda\ra%
}\tilde{\mu}^{\Lambda}(t,.,d\eta)\rceil_{\IT^{d}\times
R_{\Lambda}}.
\]
The ``distributions'' $\tilde{\mu}_{\Lambda}$ and
$\tilde{\mu}^{\Lambda}$ have the following properties~:\smallskip

\begin{itemize}
\item $\tilde{\mu}_{\Lambda}(t,dx,d\xi,d\eta)$ is in
$L^{\infty}\left(  \mathbb{R};\left(  \mathcal{S}_{\Lambda}%
^{1}\right)  ^{\prime}\right)  $;\smallskip

\item $\int_{\la \Lambda\ra}\tilde{\mu}_{\Lambda}(t,.,d\eta)$ is in
$L^{\infty }(\IR,\cM_{+}(T^{\ast}\IT^{d}))$;\smallskip

\item
 for $a\in \cS_\Lambda^1$, we have
$$ \int_{T^*\mathbb{T}^{d} \times\left\langle \Lambda
\right\rangle }a\left(  x,\xi,\eta\right)  \tilde{\mu}_{\Lambda}\left(
t,dx,d\xi,d\eta\right) \\
  = \Tr\left(  \int_{T^{\ast}\mathbb{T}_{\Lambda^{\perp}}} \Op_{1}^{\Lambda
}\left(  a\left( \cdot , \sigma,\cdot\right)  \right)  \tilde{\rho
}_{\Lambda}\left(  t,ds,d\sigma\right)  \right) $$
where $\tilde{\rho
}_{\Lambda}\left(  t\right)$ is a positive measure on $T^{\ast}\mathbb{T}_{\Lambda^{\perp}}$,
taking values in $\cL^1(T_\Lambda(L^2_\Lambda(\T^d)))$, invariant under the geodesic flow $(s, \sigma)\mapsto (s+\tau\sigma, \sigma)$ ($\tau\in\R$).
\end{itemize}

In addition,\smallskip

\begin{itemize}
\item for $a\in\cS_{\Lambda}^{1}$,
$\la\tilde{\mu}^{\Lambda}(t,dx,d\xi ,d\eta),a(x,\xi,\eta)\ra$ is
obtained as the limit of
\[
\left\langle w_{h,R}^{\Lambda}\left(  t\right)  ,a\right\rangle
:=\int _{T^{\ast}\mathbb{T}^{d}}\left(  1-\chi\left(
\frac{P_{\Lambda}\left( \xi\right)  }{Rh}\right)  \right)  a\left(
x,\xi,\frac{P_{\Lambda}\left( \xi\right)  }{h}\right)  w_{h}\left(
t\right)  \left(  dx,d\xi\right)  ,
\]
where the weak-$\ast$ limit holds in
$L^{\infty}(\IR,\cS_{\Lambda}^{1^{\prime}})$, as $h\To0$ then
$R\To+\infty$ (along subsequences);\smallskip

\item $\tilde{\mu}^{\Lambda}(t,dx,d\xi,d\eta)$ is in $L^{\infty}%
(\IR,\cM_{+}(T^{\ast}\mathbb{T}^{d}\times
\mathbb{S}_{\la\Lambda\ra} ))$;\smallskip

\item $\tilde{\mu}^{\Lambda}$ is invariant by the two flows, $\phi_{\tau}%
^{0}:(x,\xi,\eta)\mapsto(x+\tau\xi,\xi,\eta)$, and
$\phi_{\tau}^{1}:(x,\xi ,\eta)\mapsto(x+\tau\frac{\eta}{|\eta|},\xi,\eta)$ ($\tau\in\R$).
\end{itemize}

This can be considered as the first step of an induction
procedure, the $k$-th step of which will read as follows~:
\medskip

\noindent\textbf{Step $k$ of the induction~:} At step $k$, we have
decomposed $\mu(t,.)$ as a sum
\[
\mu(t,.)=\sum_{1\leq l\leq k}\:\sum_{\Lambda_{1}\supset\Lambda_{2}%
\supset\ldots\supset\Lambda_{l}}\mu_{\Lambda_{l}}^{\Lambda_{1}\Lambda
_{2}\ldots\Lambda_{l-1}}(t,.)+\sum_{\Lambda_{1}\supset\Lambda_{2}\supset
\ldots\supset\Lambda_{k}}\mu^{\Lambda_{1}\Lambda_{2}\ldots\Lambda_{k}}(t,.),
\]
where the sums run over the \emph{strictly decreasing} sequences
of primitive submodules of $\IZ^{d}$ (of lengths $l\leq k$ in the
first term, of length $k$ in the second term). These measures themselves are obtained as
\[
\mu_{\Lambda_{l}}^{\Lambda_{1}\Lambda_{2}\ldots\Lambda_{l-1}}(t,.)=\int
_{R_{\Lambda_{2}}(\Lambda_{1})\times\ldots\times
R_{\Lambda_{l}}(\Lambda
_{l-1})\times\la\Lambda_l\ra}\tilde{\mu}_{\Lambda_{l}}^{\Lambda_{1}\Lambda_{2}%
\ldots\Lambda_{l-1}}(t,.,d\eta_{1},\ldots,d\eta_{l})\rceil_{\IT^{d}\times
R_{\Lambda_{1}}},
\]%
\[
\mu^{\Lambda_{1}\Lambda_{2}\ldots\Lambda_{k}}(t,.)=\int_{R_{\Lambda_{2}%
}(\Lambda_{1})\times\ldots\times R_{\Lambda_{k}}(\Lambda_{k-1})\times \la\Lambda_k\ra%
}\tilde{\mu}^{\Lambda_{1}\Lambda_{2}\ldots\Lambda_{k}}(t,.,d\eta_{1}%
,\ldots,d\eta_{k})\rceil_{\IT^{d}\times R_{\Lambda_{1}}},
\]
where we denoted
$R_{\Lambda}(\Lambda^{\prime}):=\Lambda^{\bot}\cap
\la\Lambda^{\prime}\ra\cap\Omega_{\operatorname*{rk}\Lambda^{\prime
}-\operatorname*{rk}\Lambda}$, for
$\Lambda\subset\Lambda^{\prime}$.

Let us denote by $\cS^{k}_{\Lambda_{1},\ldots, \Lambda_{k}}$ the class of
smooth
functions $a(x, \xi, \eta_{1},\ldots, \eta_{k})$ on $T^{*}\IT^{d}%
\times\la\Lambda_{1}\ra\times\ldots\times \la\Lambda_{k}\ra$ that are (i) smooth and compactly supported
in $(x, \xi)\in T^{*}\IT^{d}$; (ii) homogeneous of degree $0$ at
infinity in each variable $\eta_{1}, \ldots, \eta_{k}$; (iii) such
that their non-vanishing $x$-Fourier coefficients correspond to
frequencies in $\Lambda_{k}$.

The ``distributions''
$\tilde{\mu}_{\Lambda_{l}}^{\Lambda_{1}\Lambda_{2}\ldots
\Lambda_{l-1}}$ and
$\tilde{\mu}^{\Lambda_{1}\Lambda_{2}\ldots\Lambda_{k}}$ have the
following properties~:

\begin{itemize}
\item
$\tilde{\mu}_{\Lambda_{l}}^{\Lambda_{1}\Lambda_{2}\ldots\Lambda_{l-1}}$
is in $L^{\infty}\left(  \IR,(\cS^{l }_{\Lambda_{1},\ldots, \Lambda_{l}})^\prime\right)    $. With respect to the variables $\eta_j\in\la\Lambda_{j}\ra$, $j=1,\ldots, l-1$, it is $0$-homogeneous and supported at infinity. Thus, (as in footnote \ref{f:homogen}) we may identify it with a distribution on the unit sphere $\mathbb{S}_{\la{\Lambda_{1}}\ra}\times\ldots\times \mathbb{S}_{\la\Lambda_{l-1}\ra}$ ;\smallskip

\item $\int_{\la\Lambda_l\ra}\tilde{\mu}_{\Lambda_{l}}^{\Lambda_{1}\Lambda_{2}%
\ldots\Lambda_{l-1}}(t,.,d\eta_{l})$ is in
$L^{\infty}(\IR,\cM_{+}(T^{\ast
}\IT^{d}\times\mathbb{S}_{\la{\Lambda_{1}}\ra}\times\ldots\times \mathbb{S}_{\la\Lambda_{l-1}\ra}))$;\smallskip

\item  for $a\in \cS^{l }_{\Lambda_{1},\ldots, \Lambda_{l}}$, we have
\begin{align}
&  \int_{T^*\mathbb{T}^{d}\times  {\la\Lambda_{1}%
\ra} \times\ldots\times {\la\Lambda_{l-1}\ra} }a\left(  x,\xi, \eta_1,\ldots, \eta_l\right) \tilde{\mu}_{\Lambda_{l}}^{\Lambda_{1}\Lambda_{2}\ldots\Lambda_{l-1}}\left(
t,dx,d\xi,d\eta_1,\ldots, d\eta_l\right) \\
& \label{e:fullintrinsic} = \Tr\left(  \int_{T^{\ast}\mathbb{T}_{\Lambda_l^{\perp}}\times  \mathbb{S}_{\la\Lambda_{1}%
\ra} \times\ldots\times \mathbb{S}_{\la\Lambda_{l-1}\ra}} \Op_{1}^{\Lambda_l
}\left(  a\left( \cdot , \sigma,  \infty\eta_1,\ldots, \infty \eta_{l-1}, \cdot\right)  \right)  \tilde\rho_{\Lambda_{l}}^{\Lambda_{1}\Lambda_{2}\ldots\Lambda_{l-1}}\left(  t,ds,d\sigma, d\eta_1, \ldots, d\eta_{l-1}\right)  \right) \end{align}
where
$\tilde\rho_{\Lambda_{l}}^{\Lambda_{1}\Lambda_{2}\ldots\Lambda_{l-1}}(t)$
is a positive measure on $T^*\T_{\Lambda_{l}^{\bot}}\times  \mathbb{S}_{\la\Lambda_{1}%
\ra} \times\ldots\times \mathbb{S}_{\la\Lambda_{l-1}\ra} $, taking values in
$\cL^{1}(T_{\Lambda_l}L^{2}_{\Lambda_l}(\IT^d))$. It is invariant under the flows
$(s, \sigma, \eta_1,\ldots, \eta_{l-1})\mapsto (s+\tau\sigma, \sigma, \eta_1,\ldots, \eta_{l-1})$ and 
$(s, \sigma, \eta_1,\ldots, \eta_{l-1})\mapsto (s+\tau\frac{\eta_j}{|\eta_j|}, \sigma, \eta_1,\ldots, \eta_{l-1})$ ($\tau\in\R$, $j=1,\ldots, l-1$).
Equation \eqref{e:fullintrinsic} implies that the projection of $ \tilde{\mu}_{\Lambda_{l}}^{\Lambda_{1}\Lambda_{2}\ldots\Lambda_{l-1}}$ on $\T^d$ is absolutely continuous.
\end{itemize}

\begin{itemize}
\item For $a\in\cS^{k}_{\Lambda_{1},\ldots, \Lambda_{k}}$,
$\la\tilde{\mu}^{\Lambda_{1}\Lambda
_{2}\ldots\Lambda_{k}}(t,dx,d\xi,d\eta_{1},\ldots,d\eta_{k}),a(x,\xi,\eta
_{1},\ldots,\eta_{k})\ra$ is obtained as the limit of
\[
\left\langle w_{h}(t,dx,d\xi),a\left(  x,\xi,\frac{P_{\Lambda_{1}}\xi}%
{h},\cdots,\frac{P_{\Lambda_{k}}\xi}{h}\right)  \left(
1-\chi\left( \frac{P_{\Lambda_{1}}\xi}{R_{1}h}\right)  \right)
\ldots\left(  1-\chi\left(
\frac{P_{\Lambda_{k}}\xi}{R_{k}h}\right)  \right)  \right\rangle .
\]
The weak limit holds in
$L^{\infty}(\IR,(\cS^{k}_{\Lambda_{1},\ldots, \Lambda_{k}})^{\prime})$, as $h\To0$
then $R_{1}\To+\infty$,..., $R_{k}\To+\infty$ (along
subsequences);\smallskip

\item $\tilde{\mu}^{\Lambda_{1}\Lambda_{2}\ldots\Lambda_{k}}$ is
in $L^{\infty}(\IR,\cM_{+}(T^{\ast}\IT^{d}\times \mathbb{S}_{\la\Lambda_1\ra}\times\ldots \mathbb{S}_{\la\Lambda_k\ra} )) $;\smallskip

\item $\tilde{\mu}^{\Lambda_{1}\Lambda_{2}\ldots\Lambda_{k}}$ is
invariant by
the $k+1$ flows, $\phi_{\tau}^{0}:(x,\xi,\eta)\mapsto(x+\tau\xi,\xi,\eta_{1}%
,\ldots,\eta_{k})$, and $\phi_{\tau}^{j}:(x,\xi,\eta_{1},\ldots,\eta_{k}%
)\mapsto(x+\tau\frac{\eta_{j}}{|\eta_j|},\xi,\eta_{1},\ldots,\eta_{k})$ (where
$j=1,\ldots,k$, $\tau\in\R$).
\end{itemize}

\medskip

\noindent\textbf{How to go from step $k$ to step
$k+1$.} 

The term $\sum_{1\leq l\leq k}%
\sum_{\Lambda_{1}\supset\Lambda_{2}\supset\ldots\supset\Lambda_{l}}%
\mu_{\Lambda_{l}}^{\Lambda_{1}\Lambda_{2}\ldots\Lambda_{l-1}}$
remains untouched after step $k$. 

To decompose further the term
$\sum_{\Lambda
_{1}\supset\Lambda_{2}\supset\ldots\supset\Lambda_{k}}\mu^{\Lambda_{1}%
\Lambda_{2}\ldots\Lambda_{k}}$, we proceed as follows. Using the
positivity of
$\tilde{\mu}^{\Lambda_{1}\Lambda_{2}\ldots\Lambda_{k}}$, we use
the procedure described in Section \ref{s:decompo} to write
\[
\tilde{\mu}^{\Lambda_{1}\Lambda_{2}\ldots\Lambda_{k}}=\sum_{\Lambda
_{k+1}\subset\Lambda_{k}}\tilde{\mu}^{\Lambda_{1}\Lambda_{2}\ldots\Lambda_{k}%
}\rceil_{\eta_{k}\in R_{\Lambda_{k+1}}(\Lambda_{k})},
\]
where the sum runs over all primitive submodules $\Lambda_{k+1}$
of $\Lambda_{k}$. Moreover, by the proof of Lemma
\ref{LemmaFourierInvariant}, all the $x$-Fourier modes of
$\tilde{\mu}^{\Lambda_{1}\Lambda_{2}\ldots
\Lambda_{k}}\rceil_{\eta_{k}\in R_{\Lambda_{k+1}}(\Lambda_{k})}$
are in $\Lambda_{k+1}$. To generalize the analysis of Section
\ref{s:second}, we consider test functions
$a\in\cS_{\Lambda_1,\ldots\Lambda_{k+1}}^{k+1}$. For such a function $a$, we let
\begin{multline*}
\left\langle
w_{h,R_{1},\ldots,R_{k}}^{\Lambda_{1}\Lambda_{2}\ldots
\Lambda_{k+1}}\left(  t\right)  ,a\right\rangle \\
:=\int_{T^{\ast}\mathbb{T}^{d}}\left(  1-\chi\left(  \frac{P_{\Lambda_{1}%
}\left(  \xi\right)  }{R_{1}h}\right)  \right)  \ldots\left(
1-\chi\left( \frac{P_{\Lambda_{k}}\left(  \xi\right)
}{R_{k}h}\right)  \right)  \left( 1-\chi\left(
\frac{P_{\Lambda_{k+1}}\left(  \xi\right)  }{R_{k+1}h}\right)
\right) \\
a\left(  x,\xi,\frac{P_{\Lambda_{1}}\left(  \xi\right)
}{h},\cdots ,\frac{P_{\Lambda_{k+1}}\left(  \xi\right)
}{h}\right)  w_{h}\left( t\right)  \left(  dx,d\xi\right)  ,
\end{multline*}
and%
\begin{multline*}
\left\langle
w_{\Lambda_{k+1},h,R_{1},\ldots,R_{k}}^{\Lambda_{1}\Lambda
_{2}\ldots\Lambda_{k}}\left(  t\right)  ,a\right\rangle \\
:=\int_{T^{\ast}\mathbb{T}^{d}}\left(  1-\chi\left(  \frac{P_{\Lambda_{1}%
}\left(  \xi\right)  }{R_{1}h}\right)  \right)  \ldots\left(
1-\chi\left( \frac{P_{\Lambda_{k}}\left(  \xi\right)
}{R_{k}h}\right)  \right)
\chi\left(  \frac{P_{\Lambda_{k+1}}\left(  \xi\right)  }{R_{k+1}h}\right) \\
a\left(  x,\xi,\frac{P_{\Lambda_{1}}\left(  \xi\right)
}{h},\cdots ,\frac{P_{\Lambda_{k+1}}\left(  \xi\right)
}{h}\right)  w_{h}\left( t\right)  \left(  dx,d\xi\right)  .
\end{multline*}
By the Calder\'{o}n-Vaillancourt theorem, both $w_{\Lambda_{k+1}%
,h,R_{1},\ldots,R_{k}}^{\Lambda_{1}\Lambda_{2}\ldots\Lambda_{k}}$
and
$w_{h,R_{1},\ldots,R_{k}}^{\Lambda_{1}\Lambda_{2}\ldots\Lambda_{k+1}}$
are bounded in
$L^{\infty}(\IR,(\cS_{\Lambda_1,\ldots,\Lambda_{k+1}}^{k+1})^{\prime}).$ After
extracting subsequences, we can take the following
limits~:
\[
\lim_{R_{k+1}\To+\infty}\cdots\lim_{R_{1}\To+\infty}\lim_{h\To0}\left\langle
w_{h,R_{1},\ldots,R_{k}}^{\Lambda_{1}\Lambda_{2}\ldots\Lambda_{k+1}}\left(
t\right)  ,a\right\rangle =:\left\langle
\tilde{\mu}^{\Lambda_{1}\Lambda
_{2}\ldots\Lambda_{k+1}},a\right\rangle ,
\]
and
\[
\lim_{R_{k+1}\To+\infty}\cdots\lim_{R_{1}\To+\infty}\lim_{h\To0}\left\langle
w_{\Lambda_{k+1},h,R_{1},\ldots,R_{k}}^{\Lambda_{1}\Lambda_{2}\ldots
\Lambda_{k}}\left(  t\right)  ,a\right\rangle =:\left\langle
\tilde{\mu
}_{\Lambda_{k+1}}^{\Lambda_{1}\Lambda_{2}\ldots\Lambda_{k}},a\right\rangle
.
\]

By the arguments of \S \ref{s:second}, one then shows that $\tilde{\mu}^{\Lambda_{1}\Lambda
_{2}\ldots\Lambda_{k+1}}$ and $\tilde{\mu
}_{\Lambda_{k+1}}^{\Lambda_{1}\Lambda_{2}\ldots\Lambda_{k}}$ satisfy all of the induction hypotheses at step $k+1$. In particular, we obtain the following analogues of Theorems \ref{thm 1st2micro} and \ref{Thm Properties}.

\begin{theorem}
(i)  $\tilde\mu^{\Lambda_{1}\Lambda_{2}\ldots\Lambda_{k+1}}\left(
t,\cdot\right) $ is positive, zero-homogeneous in the
variables $\eta_{1} \in\left\langle \Lambda_{1}\right\rangle ,
\ldots, \eta_{k+1} \in\left\langle
\Lambda_{k+1}\right\rangle $, and supported at infinity. It can thus be identified with a positive measure on $T^*\T^d\times \mathbb{S}_{\la\Lambda_1\ra}\times\ldots\times \mathbb{S}_{\la\Lambda_{k+1}\ra}$.

 $\tilde\mu_{\Lambda_{k+1}%
}^{\Lambda_{1}\Lambda_{2}\ldots\Lambda_{k}}\left(  t,\cdot\right)
$ is zero-homogeneous in the
variables $\eta_{1} \in\left\langle \Lambda_{1}\right\rangle ,
\ldots, \eta_{k} \in\left\langle
\Lambda_{k}\right\rangle $, and supported at infinity. It can thus be identified with a distribution on $T^*\T^d\times \mathbb{S}_{\la\Lambda_1\ra}\times\ldots\times \mathbb{S}_{\la\Lambda_{k}\ra}\times \la \Lambda_{k+1}\ra$.

 The projection of $\tilde\mu_{\Lambda_{k+1}%
}^{\Lambda_{1}\Lambda_{2}\ldots\Lambda_{k}}\left(  t,\cdot\right)
$ on $T^*\T^d\times \mathbb{S}_{\la\Lambda_1\ra}\times\ldots\times \mathbb{S}_{\la\Lambda_{k}\ra} $ is positive.

(ii) For a.e. $t\in\mathbb{R}$,  
$\tilde\mu^{\Lambda_{1}\Lambda _{2}\ldots\Lambda_{k+1}}\left(
t,\cdot\right)  $ and $\tilde\mu_{\Lambda_{k+1}}%
^{\Lambda_{1}\Lambda_{2}\ldots\Lambda_{k}}(t, .)$ satisfy the invariance
properties:

\[
\left(  \phi_{\tau}^{j}\right)  _{\ast} \tilde\mu_{\Lambda_{k+1}}%
^{\Lambda_{1}\Lambda_{2}\ldots\Lambda_{k}}(t, .)  =\tilde\mu_{\Lambda_{k+1}}%
^{\Lambda_{1}\Lambda_{2}\ldots\Lambda_{k}}(t, .)  ,
\]
\[
\left(  \phi_{\tau}^{j}\right)  _{\ast} \tilde\mu^{\Lambda_{1}\Lambda_{2}%
\ldots\Lambda_{k+1}}\left(  t,\cdot\right)  = \tilde\mu^{\Lambda_{1}%
\Lambda_{2}\ldots\Lambda_{k+1}}\left(  t,\cdot\right)  ,
\]

for $j=0, \ldots, k$, $\tau\in\R$.

(iii) Let
\[
\mu_{\Lambda_{k+1}}^{\Lambda_{1}\Lambda_{2}\ldots\Lambda_{k}}(t,
.)=\int_{R_{\Lambda_{2}}(\Lambda_{1})\times\ldots\times R_{\Lambda_{k+1}%
}(\Lambda_{k})\times\la \Lambda_{k+1}\ra}\tilde\mu_{\Lambda_{k+1}}^{\Lambda_{1}\Lambda
_{2}\ldots\Lambda_{k}}(t, ., d\eta_{1}, \ldots, d\eta_{k+1})\rceil
_{(x, \xi)\in \IT^{d}\times R_{\Lambda_{1}}},
\]
\[
\mu^{\Lambda_{1}\Lambda_{2}\ldots\Lambda_{k+1}}(t, .)=\int_{R_{\Lambda_{2}%
}(\Lambda_{1})\times\ldots\times R_{\Lambda_{k+1}}(\Lambda_{k})\times\la \Lambda_{k+1}\ra%
}\tilde\mu^{\Lambda_{1}\Lambda_{2}\ldots\Lambda_{k+1}}(t, .,
d\eta_{1}, \ldots, d\eta_{k+1})\rceil_{(x, \xi)\in \IT^{d}\times
R_{\Lambda_{1}}}.
\]

Then both
$\mu_{\Lambda_{k+1}}^{\Lambda_{1}\Lambda_{2}\ldots\Lambda_{k}}(t,
.)$ and $\mu^{\Lambda_{1}\Lambda_{2}\ldots\Lambda_{k+1}}(t, .) $
are positive measures on $T^{\ast}\mathbb{T}^{d}$, invariant by
the geodesic flow, and
satisfy:%
\begin{equation}
\mu^{\Lambda_{1}\Lambda_{2}\ldots\Lambda_{k}}\rceil_{\eta_{k}\in
R_{\Lambda_{k+1}}(\Lambda_{k})}(t, .)= \mu_{\Lambda_{k+1}}^{\Lambda_{1}%
\Lambda_{2}\ldots\Lambda_{k}}(t,
.)+\mu^{\Lambda_{1}\Lambda_{2}\ldots \Lambda_{k+1}}(t, .).
\end{equation}

\end{theorem}

 \begin{theorem}
(i) For a.e. $t\in\mathbb{R}$,  $\tilde\mu_{\Lambda_{k+1}}%
^{\Lambda_{1}\Lambda_{2}\ldots\Lambda_{k}}(t, .)$ is supported on
$\T^d\times \Lambda_{k+1}^\bot\times \mathbb{S}_{\la\Lambda_1\ra}\times\ldots\times \mathbb{S}_{\la\Lambda_{k}\ra}\times \la \Lambda_{k+1}\ra$
and its projection on $\mathbb{T}^{d}$ is absolutely continuous with
respect to the Lebesgue measure.

(ii) The measure
$\tilde\mu^{\Lambda_{1}\Lambda _{2}\ldots\Lambda_{k+1}}\left(
t,\cdot\right)  $ satisfies the additional invariance
properties:%
\[
\left(  \phi_{\tau}^{k+1}\right)  _{\ast} \tilde\mu^{\Lambda_{1}\Lambda_{2}%
\ldots\Lambda_{k+1}}\left(  t,\cdot\right)  = \tilde\mu^{\Lambda_{1}%
\Lambda_{2}\ldots\Lambda_{k+1}}\left(  t,\cdot\right)  ,
\]
for $\tau\in\R$.
\end{theorem}

 The ideas are identical to those of Sections
\ref{s:decompo} and \ref{s:second},
and detailed proofs will be omitted. 

\begin{remark}
By construction, if $\Lambda_{k+1}=\{0\}$, we have
$\tilde\mu^{\Lambda _{1}\Lambda_{2}\ldots\Lambda_{k+1}}=0$, and
the induction stops. The measure $\mu^{\Lambda _{1}\Lambda_{2}\ldots\Lambda_{k}}_{\Lambda_{k+1}}$ is then constant in $x$.

Similarly to Remark \ref{r:nice}, one can
also see that if $\operatorname*{rk}
\Lambda_{k+1}=1$, the invariance properties of $\tilde\mu^{\Lambda_{1}%
\Lambda_{2}\ldots\Lambda_{k+1}}$ imply that it is constant in $x$.
\end{remark}

\vspace{.5cm}

{\bf Proof of 
Theorem \ref{t:main} (i) and of Corollary \ref{t:example}.} We write 
\[
\mu(t,.)=\sum_{1\leq l\leq d+1}\:\sum_{\Lambda_{1}\supset\Lambda_{2}%
\supset\ldots\supset\Lambda_{l}}\mu_{\Lambda_{l}}^{\Lambda_{1}\Lambda
_{2}\ldots\Lambda_{l-1}}(t,.) ,
\]
and we know that each term is a positive measure on $T^*\T^d$, whose projection on $\T^d$ is absolutely continuous.
This proves Theorem \ref{t:main} (i).

 Corollary \ref{t:example} is a direct consequence of Theorem \ref{t:main} (i) and
of the identity \eqref{e:proj}, with one little subtlety. Because $T^{\ast
}\mathbb{T}^{d}$ is not compact, if $w_{h}$ converges weakly-$\ast$ to $\mu$
and $\left(  \int_{0}^{1}|U_{V}(t)u_{h}(x)|^{2}dt\right)  dx$ converges
weakly-$\ast$ to a probability measure $\nu$ on $\mathbb{T}^{d}$, it does not
follow automatically that
\[
\nu=\int_{0}^{1}\int_{\R^{d}}\mu(t,\cdot,d\xi)dt.
\]
This is only true if we know a priori that $\int_{\mathbb{T}^{d}\times\R^{d}%
}\mu(t,dx,d\xi)=1$ for almost all $t$, which means that there is no escape of
mass to infinity. To check that Theorem \ref{t:main} implies Corollary
\ref{t:example}, we must explain why, for any normalized sequence $(u_{n})\in
L^{2}(\mathbb{T}^{d})$, we can find a sequence of parameters $h_{n}\To0$ such
that the sequence $w_{u_{n}}^{h_{n}}$ does not escape to infinity. Let us
choose $h_{n}$ such that
\begin{equation}
\sum_{k\in\Z^{d},\norm{k}\leq h_{n}^{-1}}|\hat{u}_{n}(k)|^{2}\Lim_{n\To+\infty
}1, \label{e:trunc}%
\end{equation}
which is always possible. If we let $\tilde{u}_{n}(x)=\sum_{k\in
\Z^{d},\norm{k}\leq h_{n}^{-1}}\hat{u}_{n}(k)\frac{e^{ik.x}}{\left(
2\pi\right)  ^{d/2}}$, equation \eqref{e:trunc} implies that $w_{\tilde{u}%
_{n}}^{h_{n}}$ has the same limit as $w_{u_{n}}^{h_{n}}$. On the other hand
$w_{\tilde{u}_{n}}^{h_{n}}$ is supported in the compact set $\mathbb{T}%
^{d}\times B(0,1)\subset\mathbb{T}^{d}\times\R^{d}$. Thus $w_{\tilde{u}_{n}%
}^{h_{n}}$ cannot escape to infinity. Let us point out that with this choice
of scale $(h_{n})$, the sequence $(u_{n})$ becomes $h_{n}$-oscillating, in the
terminology introduced in \cite{GerardMesuresSemi91, GerLeich93}.

\section{Propagation law for $\tilde\rho_\Lambda$\label{s:propagation}}

We now study how $\tilde\rho_\Lambda(t, \cdot)$ (defined in
Proposition \ref{p:weakstarlimit} \eqref{e:rholambda}) depends on
$t$. This will allow us to complete the proof of Theorem
\ref{t:precise} and will be crucial in the proof of the
observability inequality, Theorem \ref{t:obs}. We use the notation
of \S \ref{s:sy}. In particular, $s$ will always be a variable in
$\mathbb{T}_{\Lambda^{\perp}}$, and $y$ a variable in
$\mathbb{T}_{\Lambda }$.

In order to state our main result, let us introduce some notation.
Let $\widehat{V}_{k}\left(  t\right)  $, $k\in\mathbb{Z}$, denote
the Fourier coefficients of the potential $V\left(  t,\cdot\right)
$. We denote by $\left\langle V\right\rangle _{\Lambda}\left(
t,\cdot\right)  $ the average of $V\left(  t,\cdot\right)  $ along
$ \Lambda^\bot $,
in other words~:%
\[
\left\langle V\right\rangle _{\Lambda}\left(  t,\cdot\right)
:=\sum _{k\in\Lambda}\widehat{V}_{k}\left(  t\right)
\frac{e^{ik\cdot x}}{\left( 2\pi\right)  ^{d/2}}.
\]
We put $H_{\left\langle V\right\rangle _{\Lambda}}^{\Lambda}\left(
t\right) :=-\frac{1}{2}\Delta_{\Lambda}+\left\langle
V\right\rangle _{\Lambda}\left(  t,\cdot\right)  $ where $\Delta_{\Lambda}$ is the Laplacian on
$\la\Lambda\ra$, and denote by
$U_{\left\langle V\right\rangle _{\Lambda}}^{\Lambda}\left(
t\right)  $ the unitary evolution in $L^{2}\left(
\mathbb{T}_{\Lambda}\right)  $, starting at $t=0$, generated by
$H_{\left\langle V\right\rangle _{\Lambda}}^{\Lambda}\left(
t\right)  $.

\begin{proposition}\label{p:average} Let $\tilde{\rho}_{\Lambda}\in L^{\infty}\left(  \mathbb{R};\mathcal{M}%
_{+}\left(
T^{\ast}\mathbb{T}_{\Lambda^{\perp}};\mathcal{L}^{1}\left(
L^{2}\left(  \mathbb{T}_{\Lambda}\right)  \right)  \right)
\right)  $ be a limit of $\left(  n_{h}^{\Lambda}\right)  $ as in
Proposition \ref{p:weakstarlimit}.

Let  $(s, \sigma)\mapsto K( \sigma)$  be a function in $ C_{c}^{\infty}\left(  T^{\ast}\mathbb{T}_{\Lambda^{\perp}%
};\mathcal{K}\left(  L^{2}\left(  \mathbb{T}_{\Lambda}\right)
\right) \right)  $ that does not depend on $s$.

Then
\[\frac{d}{dt}
  \Tr\int_{\mathbb{T}%
_{\Lambda^{\perp}}\times R_{\Lambda}}K\left(  \sigma\right)
\tilde{\rho}_{\Lambda}\left( t,ds,d\sigma\right)
= i \Tr\int_{\mathbb{T}%
_{\Lambda^{\perp}}\times R_{\Lambda}}%
\left[  H_{\la V\ra_{\Lambda}}%
^{\Lambda}\left(  t,\cdot\right), K\left(  \sigma\right) \right]
\tilde{\rho}_{\Lambda}\left(  t,ds,d\sigma\right).
\]
\end{proposition}

\begin{corollary}\label{c:propagation} Let $\mu_{\Lambda}\left(  t,\cdot\right)  $ be the measure defined in
Theorem \ref{thm 1st2micro}.

For any $a\in C_{c}^{\infty}\left(  T^{\ast}\mathbb{T}^{d}\right)
$ with
Fourier coefficients in $\Lambda$ the following holds:%
\[
\int_{T^*\mathbb{T}^{d} }a\left(  x,\xi\right)
\mu_{\Lambda }\left(  t,dx,d\xi\right)  =\Tr\left(
\int_{\T_{\Lambda^\bot}\times R_\Lambda}U_{\left\langle
V\right\rangle _{\Lambda}}^{\Lambda}\left( t\right)
^{\ast}m_{a\circ\pi_{\Lambda}}\left(  \sigma\right)
U_{\left\langle V\right\rangle _{\Lambda}}^{\Lambda}\left(
t\right)  \tilde{\rho}_{\Lambda}\left( 0,ds,d\sigma\right) \right)
.
\]

\end{corollary}

Proposition \ref{p:average} will be a consequence of a more
general propagation law. For fixed $s
\in\mathbb{T}_{\Lambda^{\perp}}$, denote by
$U_{V}^{\Lambda}\left(  t,s\right)  $ ($t\in\R$) the propagator corresponding to
the unitary evolution on $L^{2}\left(  \mathbb{T}_{\Lambda
}\right)  $, starting at $t=0$, generated by
\[
H_{V}^{\Lambda}\left(  t,s\right)
:=-\frac{1}{2}\Delta_{\Lambda}+V\left( t,\pi_{\Lambda}\left(
s,y\right)  \right)  .
\]
Our main goal in this section will be to establish the following
result.
\begin{lemma}\label{l:Cinfty} For all $K$ as in Proposition \ref{p:average},
\[\frac{d}{dt}
  \Tr\int_{\mathbb{T}%
_{\Lambda^{\perp}}\times R_\Lambda}\left(  \sigma\right)
\tilde{\rho}_{\Lambda}\left( t,ds,d\sigma\right)
= i \Tr\int_{\mathbb{T}%
_{\Lambda^{\perp}}\times R_\Lambda}\left[  H_{V}%
^{\Lambda}\left(  t,s\right), K\left(  \sigma\right) \right]
\tilde{\rho}_{\Lambda}\left(  t,ds,d\sigma\right)
\]
(where $\frac{d}{dt}$ is interpreted in distribution sense).

%and as a consequence
%\begin{align*}
%&  \int_{\mathbb{T}^{d}\times\Lambda^{\perp}\times\left\langle \Lambda
%\right\rangle }a\left(  x,\xi,\eta\right)  \tilde{\mu}_{\Lambda}\left(
%t,dx,d\xi,d\eta\right) \\
%&  =\Tr\left(  \int_{T^{\ast}\mathbb{T}_{\Lambda^{\perp}}}U_{V}^{\Lambda
%}\left(  t,s\right)  ^{\ast}\Op_{1}^{\Lambda}\left(  a_{\Lambda}^{0}\left(
%\sigma,\cdot\right)  \right)  U_{V}^{\Lambda}\left(  t,s\right)  \tilde{\rho
%}_{\Lambda}\left(  0,ds,d\sigma\right)  \right)  .
%\end{align*}

\end{lemma}

That Proposition \ref{p:average} follows from Lemma
 \ref{l:Cinfty} is a consequence of the invariance of
$\tilde{\rho}_{\Lambda}\left( t,\cdot\right)  $ with respect to
the geodesic flow.

\begin{proof}[Proof that Lemma \ref{l:Cinfty} implies Proposition
\ref{p:average}]
Assume that Lemma \ref{l:Cinfty} holds.
%Since $\tilde\rho_\Lambda$ is a positive measure, we can restrict to $(s, \sigma)\in\mathbb{T}%
%_{\Lambda^{\perp}}\times R_\Lambda$ and write
%\[\frac{d}{dt}
%  \Tr\int_{\mathbb{T}%
%_{\Lambda^{\perp}}\times R_\Lambda}K\left(  \sigma\right)
%\tilde{\rho}_{\Lambda}\left( t,ds,d\sigma\right)
%= i \Tr\int_{\mathbb{T}%
%_{\Lambda^{\perp}}\times R_\Lambda}\left[  H_{V}%
%^{\Lambda}\left(  t,s\right), K\left(  \sigma\right) \right]
%\tilde{\rho}_{\Lambda}\left(  t,ds,d\sigma\right).
%\]
 Since $\tilde{\rho}_{\Lambda}\left(  t,\cdot\right)  $ is invariant by $s\mapsto s+\tau\sigma$ ($\tau\in\R$), it follows from Lemma \ref{LemmaFourierInvariant} that
$\tilde{\rho}_{\Lambda}\left(  t,\cdot\right)  \rceil_{\mathbb{T}%
_{\Lambda^{\perp}}\times R_{\Lambda}}$ is invariant by all
translations $s\mapsto s+v$ with $v\in\Lambda^{\perp}$. Therefore,
\[
\tilde{\rho}_{\Lambda}\left(  t,\cdot\right) \rceil_{\mathbb{T}%
_{\Lambda^{\perp}}\times R_{\Lambda}}  =ds\otimes\int_{\T_{\Lambda^{\perp}}%
}\tilde{\rho}_{\Lambda}\left(  t,ds,\cdot\right)
\rceil_{R_{\Lambda}} .
\]
As
\[
\int_{\T_{\Lambda^{\perp}}}H_{V}^{\Lambda}\left(  t,s\right)  ds=-\frac{1}{2}%
\Delta_{\Lambda}+\int_{\T_{\Lambda^{\perp}}}V\left(
t,\pi_{\Lambda}\left( s,y\right)  \right)  ds=H_{\left\langle
V\right\rangle _{\Lambda}}^{\Lambda }\left(  t\right)  ,
\]
the result follows.
\end{proof}
Next we shall prove Lemma \ref{l:Cinfty}, first in the smooth
case, then for continuous potentials and finally for potentials
that satisfy assumption (R).

\subsection{The case of a $C^\infty$ potential}
Here we shall assume that $V \in
C^{\infty}(\mathbb{R}\times\mathbb{T}^{d})$. The restriction of
$n_{h}^{\Lambda}\left( t\right) $ to the class of test functions
that do not depend on $s\in\mathbb{T}_{\Lambda^{\perp}}$ satisfies
a certain propagation law, that we now describe. This generalizes
statement (ii) in Theorem 2 of \cite{MaciaAv}.

\begin{lemma}\label{LemmaFinalEgorov}
If $K\in C_{c}^{\infty}\left(  \Lambda^{\perp};\mathcal{K}\left(  L^{2}\left(
\mathbb{T}_{\Lambda}\right)  \right)  \right)  $ is a function that does not
depend on $s$ then%
\begin{equation}\label{e:derivative}\frac{d}{dt}
\left\langle n_{h}^{\Lambda}\left(  t\right)  ,K\right\rangle = i\left\langle T_{\Lambda}u_{h},\left[  H_{V}%
^{\Lambda}\left(  t,\cdot\right), K\left(  hD_{s}\right)  \right] T_{\Lambda}u_{h}\right\rangle _{L^{2}\left(  \mathbb{T}%
_{\Lambda^{\perp}};L^{2}\left(  \mathbb{T}_{\Lambda}\right)  \right)  }.
\end{equation}

\end{lemma}

%\begin{lemma}
%Let $a\in\mathcal{S}_{\Lambda}^{1}$; set $a_{R}\left(
%x,\xi,\eta\right)  :=\chi\left(  \eta/R\right)  a\left(  x,\xi,\eta\right)  $
%and define $a_{R,\Lambda}^{h}\in C_{c}^{\infty}\left(  \Lambda^{\perp}\times
%T^{\ast}\mathbb{T}_{\Lambda}\right)  $ by%
%\[
%a_{R,\Lambda}^{h}\left(  \sigma,y,\eta\right)  :=a_{R}\left(  \tilde{\pi
%}_{\Lambda}\left(  \sigma,y,h\eta\right)  ,\eta\right)  ,\quad\left(
%y,\eta\right)  \in T^{\ast}\mathbb{T}_{\Lambda},\quad\sigma\in\Lambda^{\perp
%}.
%\]
%Then:%
%\[
%\frac{d}{dt}\left\langle w_{\Lambda,h,R}\left(  t\right)  ,a\right\rangle
%=\frac{i}{p_\Lambda}\left\langle T_{\Lambda}u_{h}\left(  t,\cdot\right)  ,\left[  H_{V}%
%^{\Lambda}\left(  t,\cdot\right)  ,\Op_{1}^{\Lambda}\left(  a_{R,\Lambda}%
%^{h}\left(  hD_{s},\cdot\right)  \right)  \right]  T_{\Lambda}u_{h}\left(
%t,\cdot\right)  \right\rangle ,
%\]

%where $p_{\Lambda}\in\mathbb{N}$ is the degree of $\pi_{\Lambda}$.
%\end{lemma}

\begin{proof}
 It is simple to check that
(\ref{ChangeCoord}) gives:%
\[
T_{\Lambda}\Delta T_{\Lambda}^{\ast}=\Delta_{\Lambda}+\Delta_{\Lambda^{\perp}%
}.
\]%

Moreover, it is clear that:%
\[
\left[  \Delta_{\Lambda^{\perp}},K\left(  hD_{s}\right)   \right]  =0.
\]
Therefore, equation (\ref{e:nh}), in the case when $K$ does not depend on $s$, gives \eqref{e:derivative}.%
%\[\frac{d}{dt}
%\left\langle n_{h}^{\Lambda}\left(  t\right)  ,K\right\rangle = \frac
%{i}{p_{\Lambda}}\left\langle T_{\Lambda}u_{h}(t),\left[  H_{V}%
%^{\Lambda}\left(  t,\cdot\right), K\left(  hD_{s}\right)  \right] T_{\Lambda}u_{h}(t)\right\rangle _{L^{2}\left(  \mathbb{T}%
%_{\Lambda^{\perp}};L^{2}\left(  \mathbb{T}_{\Lambda}\right)  \right)  }.
%\]
%which is our claim.
\end{proof}

Taking limits in equation \eqref{e:derivative} and taking into
account that we can restrict  $\tilde\rho_\Lambda$ to $(s,
\sigma)\in\mathbb{T}_{\Lambda^{\perp}}\times R_\Lambda$ (since
 it is a positive measure), concludes the proof of Lemma
\ref{l:Cinfty} in this case.

\subsection{The case of a continuous potential\label{s:continuous}}
In this section, we assume that $V\in C(\R\times\T^d)$. In this
case, Lemma \ref{LemmaFinalEgorov} still holds, but we cannot
obtain \ref{l:Cinfty} by simply taking limits. Instead, we
shall use an elementary approximation argument.

We introduce a sequence $V_n$ of $C^\infty$ potentials, such that
$$\norm{V-V_n}_{L^\infty}\leq \frac1n.$$
We rewrite equation \eqref{e:derivative},

\begin{multline*} \frac{d}{dt}
\left\langle n_{h}^{\Lambda}\left(  t\right)  ,K\right\rangle =
i\left\langle T_{\Lambda}u_{h}(t),\left[  H_{V_n}%
^{\Lambda}\left(  t,\cdot\right), K\left(  hD_{s}\right)  \right]
T_{\Lambda}u_{h}(t)\right\rangle  \\+ i\left\langle
T_{\Lambda}u_{h}(t),\left[  V-V_n, K\left(  hD_{s}\right)  \right]
T_{\Lambda}u_{h}(t)\right\rangle.
\end{multline*}
%\texttt{[Nalini : are the notations OK ? maybe we should apply the change of variable $\pi_\Lambda$ everywhere to the potential $V$, but it would be very heavy.]}
We use the inequality
\[
\left|\left\langle T_{\Lambda}u_{h},\left[  V-V_n, K\left(  hD_{s}\right)  \right] T_{\Lambda}u_{h}\right\rangle\right|\leq 2\norm{V-V_n}_{L^\infty}\sup_{\sigma\in \Lambda^\bot}\norm{K(\sigma)}
\]
to estimate the error when replacing $V$ by $V_n$.

In the limit $h\To 0$,
\[
\left\langle T_{\Lambda}u_{h},\left[  H_{V_n}%
^{\Lambda}\left(  t,\cdot\right), K\left(  hD_{s}\right)  \right] T_{\Lambda}u_{h}\right\rangle
\To
\Tr\int_{T^{\ast}\mathbb{T}_{\Lambda^{\perp}}%
}\left[  H_{ V_n}%
^{\Lambda}\left(  t,\cdot\right), K\left(  \sigma\right) \right] \tilde{\rho}_{\Lambda}\left(  t,ds,d\sigma\right)
\]
since $V_n$ is smooth.
We use again the inequality
\[\left|
\Tr\int_{T^{\ast}\mathbb{T}_{\Lambda^{\perp}}%
}\left[ V-V_n
 , K\left(  \sigma\right) \right] \tilde{\rho}_{\Lambda}\left(  t,ds,d\sigma\right)
\right|\leq 2
\norm{V-V_n}_{L^\infty}\sup_{\sigma\in \Lambda^\bot}\norm{K(\sigma)}
\]
to estimate the error when replacing $V_n$ by $V$.

Letting $h\To 0$ and then $n\To +\infty$, we find that
\[\frac{d}{dt}
  \Tr\int_{T^{\ast}\mathbb{T}_{\Lambda^{\perp}}%
}K\left(  \sigma\right)  \tilde{\rho}_{\Lambda}\left(  t,ds,d\sigma\right)
= i \Tr\int_{T^{\ast}\mathbb{T}_{\Lambda^{\perp}}%
}\left[  H_{V}%
^{\Lambda}\left(  t,s\right), K\left(  \sigma\right) \right]
\tilde{\rho}_{\Lambda}\left(  t,ds,d\sigma\right)
\]
where $\frac{d}{dt}$ is meant in the distribution sense.

Again, we can restrict $\tilde\rho_\Lambda$ to $(s,
\sigma)\in\mathbb{T}_{\Lambda^{\perp}}\times R_\Lambda$ since
 it is a positive measure.
This concludes the proof of Lemma \ref{l:Cinfty} in the continuous
case.

\subsection{Case of an $L^\infty$ potential\label{s:linfty}}
Let us turn to the case of a potential $V$ that satisfies
condition (R) of the introduction. We use again an approximation
argument, but we have to use the fact that we already know that
the limit measures are absolutely continuous.

It is enough to consider the restriction of $n^\Lambda_h(t)$ to $t\in[0, T]$, for any arbitrary $T$.
For any $\epsilon >0$, we then consider the set $K_\epsilon$ and the function $V_\epsilon$
described in Assumption (R).
Consider an open set $W_{2\eps}$ of Lebesgue measure $<2\eps$ such that $K_\eps\subset W_{2\eps}$. Let us introduce a continuous function $\chi_\eps$ taking values in $[0, 1]$, and which takes the value $1$ on the complement of $W_{2\eps}$ and $0$ on $K_\eps$ (this is where we use the fact that $K_\eps$ is closed).

Lemma \ref{LemmaFinalEgorov} still holds.
We use it to write
\begin{multline}\label{e:deco} \frac{d}{dt}
\left\langle n_{h}^{\Lambda}\left(  t\right)  ,K\right\rangle =
i\left\langle T_{\Lambda}u_{h}(t),\left[  H_{\chi_\eps V_\eps}%
^{\Lambda}\left(  t,\cdot\right), K\left(  hD_{s}\right)  \right] T_{\Lambda}u_{h}(t)\right\rangle
\\+ i\left\langle T_{\Lambda}u_{h}(t),\left[ \chi_\eps(t)\left(V(t)-V_\eps(t)\right), K\left(  hD_{s}\right)  \right] T_{\Lambda}u_{h}(t)\right\rangle
 \\+ i\left\langle T_{\Lambda}u_{h}(t),\left[  V(1-\chi_\eps)(t), K\left(  hD_{s}\right)  \right] T_{\Lambda}u_{h}(t)\right\rangle.
\end{multline}

Arguing as in \S \ref{s:continuous}, we see that
$$\left\langle T_{\Lambda}u_{h},\left[  H_{\chi_\eps V_\eps}%
^{\Lambda}\left(  t,\cdot\right), K\left(  hD_{s}\right)  \right] T_{\Lambda}u_{h}\right\rangle$$
converges to
\begin{equation}\label{e:limit}\Tr\int_{T^{\ast}\mathbb{T}_{\Lambda^{\perp}}%
}\left[  H_{\chi_\eps V_\eps}%
^{\Lambda}\left(  t,\cdot\right), K\left(  \sigma\right) \right] \tilde{\rho}_{\Lambda}\left(  t,ds,d\sigma\right)
\end{equation}
in the limit $h\To 0$, since $\chi_\eps V_\eps$ is continuous.
Note that we can replace $V_\eps$ by $V$ in this limiting term
\eqref{e:limit}, up to an error of $2\eps
\sup_{\sigma\in\Lambda^\perp}\norm{K(\sigma)}$. Analogously, we
are going to show that in the limit $h\To 0$ the remaining error
terms give a contribution that vanishes as $\epsilon$ tends to
zero. In other words, we are going to show that the following
equation holds,
\begin{equation}\label{e:limit2}
\frac{d}{dt}
  \Tr\int_{T^{\ast}\mathbb{T}_{\Lambda^{\perp}}%
}K\left(  \sigma\right)  \tilde{\rho}_{\Lambda}\left(
t,ds,d\sigma\right)
= i \Tr\int_{T^{\ast}\mathbb{T}_{\Lambda^{\perp}}%
}\left[  H_{\chi_\eps V}%
^{\Lambda}\left(  t,s\right), K\left(  \sigma\right) \right]
\tilde{\rho}_{\Lambda}\left(  t,ds,d\sigma\right)+\sup_{\sigma\in\Lambda^\perp}\norm{K(\sigma)}R_{\epsilon},
\end{equation}
where $R_{\epsilon}$ does not depend on $K$, and goes to $0$ as $\epsilon \To 0$. To do so, we
estimate the error terms involved.

The term $\left|\left\langle T_{\Lambda}u_{h}(t),\left[
\chi_\eps(V-V_\eps), K\left(  hD_{s}\right)  \right]
T_{\Lambda}u_{h}(t)\right\rangle \right|$ is easily seen to be
bounded from above by $2\eps
\sup_{\sigma\in\Lambda^\perp}\norm{K(\sigma)}$.

We now turn to the error term involving $V(1-\chi_\eps)$ in \eqref{e:deco}. We use
the fact that this function is supported on a set of small
measure, and that we know that the limit measures are absolutely
continuous. We deal with the first term in the commutator, the
second one may be treated analogously. Clearly
$$\left|\left\langle T_{\Lambda}u_{h}(t),  V(1-\chi_\eps) K\left(  hD_{s}\right)  T_{\Lambda}u_{h}(t)\right\rangle\right| \leq
\norm{V}_{L^\infty}\sup_{\sigma\in\Lambda^\bot}\norm{K(\sigma)}
\norm{ u_{h}(t)} \norm{(1-\chi_\eps) u_{h}(t)} .$$ Integrating
along an $L^1$ function $\theta(t)$,
\begin{multline*}\left|\int_0^T \theta(t)\left\langle T_{\Lambda}u_{h}(t),  V(1-\chi_\eps) K\left(  hD_{s}\right)  T_{\Lambda}u_{h}(t)\right\rangle dt\right|\\
\leq
\norm{V}_{L^\infty}\sup_{\sigma\in\Lambda^\bot}\norm{K(\sigma)}
\int_0^T |\theta(t)|\norm{ u_{h}(t)}
\norm{(1-\chi_\eps) u_{h}(t)} dt\\
\leq
\norm{V}_{L^\infty}\sup_{\sigma\in\Lambda^\bot}\norm{K(\sigma)}
\left(\int_0^T |\theta(t)|\norm{ u_{h}(t)}^2dt
\right)^{1/2}\left(\int_0^T|\theta(t)|
\norm{(1-\chi_\eps) u_{h}(t)}^2 dt\right)^{1/2}\\
=\norm{V}_{L^\infty}\sup_{\sigma\in\Lambda^\bot}\norm{K(\sigma)}
\left(\int_0^T |\theta(t)| dt
\right)^{1/2}\left(\int_0^T|\theta(t)| \norm{(1-\chi_\eps)
u_{h}(t)}^2 dt\right)^{1/2}
\end{multline*}
By Corollary \ref{t:example} we know that $\int_0^T |\theta(t)|
\norm{(1-\chi_\eps) u_{h}(t)}^2 dt$ converges as $h\To 0$ (along a
subsequence) to
$$\int_0^T \int_{\mathbb{T}^d} |\theta(t)||1-\chi_\eps(t, x)|^2 \nu_t(dx)dt$$
where $\nu_t$ is an absolutely continuous probability measure on
$\T^d$. The function $|1-\chi_\eps(t, x)|$ takes values in $[0,
1]$ and is supported in $W_{2\eps}$, of measure $<2\eps$. Thus,
$$\int_0^T \int_{\mathbb{T}^d} |\theta(t)||1-\chi_\eps(t, x)|^2 \nu_t(dx)dt\To 0$$
as $\eps\To 0$.

Equation (\ref{e:limit2}) is now proved. Restricting  $\tilde\rho_\Lambda$ to $(s,
\sigma)\in\mathbb{T}_{\Lambda^{\perp}}\times R_\Lambda$, it follows that
\begin{equation}\label{e:limit3}
\frac{d}{dt}
  \Tr\int_{\mathbb{T}_{\Lambda^{\perp}}\times R_{\Lambda}}K\left(  \sigma\right)  \tilde{\rho}_{\Lambda}\left(
t,ds,d\sigma\right)
= i \Tr\int_{\mathbb{T}_{\Lambda^{\perp}}\times R_{\Lambda}
}\left[  H_{\chi_\eps V}%
^{\Lambda}\left(  t,s\right), K\left(  \sigma\right) \right]
\tilde{\rho}_{\Lambda}\left(  t,ds,d\sigma\right)+\sup_{\sigma\in\Lambda^\perp}\norm{K(\sigma)}R_{\epsilon},
\end{equation}
There remains to show how to conclude Lemma \ref{l:Cinfty} from
equation (\ref{e:limit3}).  To do so, we prove that
\begin{equation} \Tr\int_{\mathbb{T}%
_{\Lambda^{\perp}}\times R_{\Lambda}}\left[  H_{\chi_\eps V}%
^{\Lambda}\left(  t,\cdot\right), K\left(  \sigma\right) \right]
\tilde{\rho}_{\Lambda}\left(  t,ds,d\sigma\right)
\end{equation}
is the same as
\begin{equation} \Tr\int_{\mathbb{T}_{\Lambda^{\perp}}\times R_{\Lambda}}%
\left[  H_{ V}%
^{\Lambda}\left(  t,\cdot\right), K\left(  \sigma\right) \right]
\tilde{\rho}_{\Lambda}\left(  t,ds,d\sigma\right)
\end{equation}
up to an error which goes to $0$ with $\epsilon$.
The difference between both is
\begin{multline*} \Tr\int_{\mathbb{T}%
_{\Lambda^{\perp}}\times R_{\Lambda}}%
\left[ V(1-\chi_\eps)(t), K\left(  \sigma\right) \right]
\tilde{\rho}_{\Lambda}\left(  t,ds,d\sigma\right) =
 \Tr\int_{\mathbb{T}_{\Lambda^{\perp}}\times R_{\Lambda}}  V(1-\chi_\eps) (t)K\left(  \sigma\right)  \tilde{\rho}_{\Lambda}\left(  t,ds,d\sigma\right) \\ -
 \Tr\int_{\mathbb{T}%
_{\Lambda^{\perp}}\times R_{\Lambda}}%
  K\left(  \sigma\right) V(1-\chi_\eps)
(t)\tilde{\rho}_{\Lambda}\left(  t,ds,d\sigma\right).
\end{multline*}
 Let us consider for instance
 \begin{equation}\label{e:traceepsilon} \Tr\int_{\mathbb{T}%
_{\Lambda^{\perp}}\times R_{\Lambda} } V(1-\chi_\eps) (t)K\left(
\sigma\right) \tilde{\rho}_{\Lambda}\left( t,ds,d\sigma\right).
\end{equation}
For any $\theta\in L^1(\R)$, the measure
$$a\in C([0, T]\times \T^d)\mapsto \int_0^T\theta(t)\Tr\int_{\mathbb{T}%
_{\Lambda^{\perp}}\times R_{\Lambda} }  m_a K\left(  \sigma\right)
\tilde{\rho}_{\Lambda}\left( t,ds,d\sigma\right)dt $$ is absolutely continuous, therefore
$$\int_0^T\theta(t) \Tr\int_{\mathbb{T}%
_{\Lambda^{\perp}}\times R_{\Lambda} } V(1-\chi_\eps) (t)K\left(
\sigma\right) \tilde{\rho}_{\Lambda}\left(
t,ds,d\sigma\right)dt$$ goes to $0$ when $\epsilon\To 0$.

This finishes the proof of Lemma \ref{l:Cinfty}.

\begin{remark} The same argument applies to show that the operator-valued measure
$$\tilde\rho_{\Lambda_{l}}^{\Lambda_{1}\Lambda_{2}\ldots\Lambda_{l-1}}\left(t,
ds, d\sigma, d\eta_{1},\ldots, d\eta_{l}\right)$$
appearing in \eqref{e:fullintrinsic} satisfies the propagation law analogous to Proposition \ref{p:average}
\begin{align*}
&\frac{d}{dt}
  \Tr\int_{T^*\mathbb{T}%
_{{\Lambda^{\perp}_l}} \times R_{\Lambda_{2}}(\Lambda_{1})\times\ldots\times
R_{\Lambda_{l}}(\Lambda
_{l-1}) }K\left(  \sigma\right)
\tilde\rho_{\Lambda_{l}}^{\Lambda_{1}\Lambda_{2}\ldots\Lambda_{l-1}}\left( t,ds,d\sigma, d\eta_1,\ldots, d\eta_{l-1}\right)
\\& = i \Tr\int_{T^*\mathbb{T}%
_{{\Lambda^{\perp}_l}} \times R_{\Lambda_{2}}(\Lambda_{1})\times\ldots\times
R_{\Lambda_{l}}(\Lambda
_{l-1}) }%
\left[  H_{\la V\ra_{\Lambda_l}}%
^{\Lambda_l}\left(  t,\cdot\right), K\left(  \sigma\right) \right]
\tilde\rho_{\Lambda_{l}}^{\Lambda_{1}\Lambda_{2}\ldots\Lambda_{l-1}}\left( t,ds,d\sigma, d\eta_1,\ldots, d\eta_{l-1}\right).
\end{align*}

\end{remark}

\subsection{End of proof of Theorem \ref{t:precise}} To end the
proof of Theorem \ref{t:precise}, we let
\[
\nu_{\Lambda}(t,\cdot)=\sum_{0\leq k\leq
d-1}\sum_{\Lambda_{1}\supset
\Lambda_{2}\supset\cdots\supset\Lambda_{k}\supset\Lambda}\int_{\IR^{d}}%
\mu_{\Lambda}^{\Lambda_{1}\Lambda_{2}\ldots\Lambda_{k}}(t,\cdot,d\xi),
\]
where $\Lambda_{1},\ldots,\Lambda_{k}$ run over the set of
strictly decreasing sequences of submodules, such that
$\Lambda_{k}\subset\Lambda$. We also let
\[
\sigma_{\Lambda}=\sum_{0\leq k\leq
d-1}:\sum_{\Lambda_{1}\supset\Lambda
_{2}\supset\cdots\supset\Lambda_{k}\supset\Lambda}\int_{\T^d\times R_{\Lambda_1}\times R_{\Lambda_{2}}(\Lambda_{1})\times\ldots\times R_{\Lambda%
}(\Lambda_{k})\times\la \Lambda\ra}\tilde{\rho}_{\Lambda
}^{\Lambda_{1}\Lambda_{2}\ldots\Lambda_{k}}\left(0,ds, d\sigma,d\eta_{1},\ldots ,d\eta_{k}, d\eta\right)  ,
\]
where the
$\tilde{\rho}_{\Lambda}^{\Lambda_{1}\Lambda_{2}\ldots\Lambda_{k}}$
are the operator-valued measures appearing in \eqref{e:fullintrinsic}.

\section{Propagation of $\bar{\mu}$ and end of the proof of Theorem
\ref{t:main}\label{endthm1}}

We have already proved statement (i) of Theorem \ref{t:main}; we shall now
concentrate on (ii). We shall need a preliminary result, which is of
independent interest, that describes the propagation of $\bar{\mu},$ the
projection of $\mu$ onto the variable $\xi\in\mathbb{R}^{d}$.

\begin{proposition}
\label{p:muc}Suppose that $\mu_{0}\in\mathcal{M}_{+}\left(  T^{\ast}%
\mathbb{T}^{d}\right)  $ is a semiclassical measure of $\left(  u_{h}\right)
$. Then $\bar{\mu}$ is constant for a.e. $t$ and,%
\begin{equation}
\bar{\mu}=\int_{\mathbb{T}^{d}}\mu_{0}\left(  dy,\cdot\right)
.\label{e:margm0}%
\end{equation}

\end{proposition}

\begin{proof}
We write for $a\in C_{c}^{\infty}(\R^{d})$ and $T\in\mathbb{R}$:
\begin{multline*}
\la U_{V}(T)u_{h},a\left(  hD_{x}\right)  U_{V}(T)u_{h}\ra-\la u_{h},a\left(
hD_{x}\right)  u_{h}\ra\\
=-i\int_{0}^{T}\la U_{V}(t)u_{h},\left[  a\left(  hD_{x}\right)  ,-\frac
{\Delta}{2}+V\right]  U_{V}(t)u_{h}\ra dt=-i\int_{0}^{T}\la U_{V}%
(t)u_{h},\left[  a\left(  hD_{x}\right)  ,V\right]  U_{V}(t)u_{h}\ra dt.
\end{multline*}
If $V\in C^{\infty}(\R\times\T^{d})$, we have the estimate coming from
pseudodifferential calculus,
\[
\norm{\left[a(hD_{x}),  V\right]}_{L^{2}(\T^{d})\To L^{2}(\T^{d})}=\cO(h).
\]
This implies that, for every $T\in\mathbb{R}$:%
\begin{equation}
\lim_{h\rightarrow0^{+}}\la U_{V}(T)u_{h},a\left(  hD_{x}\right)
U_{V}(T)u_{h}\ra=\int_{T^{\ast}\mathbb{T}^{d}}a\left(  \xi\right)  \mu
_{0}\left(  dx,d\xi\right)  , \label{e:limm}%
\end{equation}
which in turn shows (\ref{e:margm0}).

When $V\in C(\R\times\T^{d})$, we establish (\ref{e:limm}) by showing that
\[
\norm{\left[a(hD_{x}),  V\right]}_{L^{2}(\T^{d})\To L^{2}(\T^{d})}%
\Lim_{h\To0}0.
\]
This can be proved by an approximation argument as in \S \ref{s:continuous}~:
\[
\lbrack a\left(  hD_{x}\right)  ,V]=[a\left(  hD_{x}\right)  ,V_{n}]+[a\left(
hD_{x}\right)  ,V-V_{n}],
\]
with $[a\left(  hD_{x}\right)  ,V_{n}]\Lim_{h\To0}0$ if $V_{n}\in C^{\infty
}(\R\times\T^{d})$, and
\[
\norm{[a(hD_{x}), V- V_n]}_{L^{2}\To L^{2}}\leq2\norm{a(hD_{x})}_{L^{2}\To
L^{2}}\norm{V-V_n}_{L^{\infty}}.
\]
If $V$ satisfies Assumption (R), we write with the same notation as in
\S \ref{s:linfty},
\begin{multline*}\int_{0}^{T}\la U_{V}(t)u_{h},[a\left(  hD_{x}\right)  ,V]U_{V}%
(t)u_{h}\ra dt\\
=\int_{0}^{T}\la U_{V}(t)u_{h},[a\left(  hD_{x}\right)  ,V_{\epsilon}%
\chi_{\epsilon}]U_{V}(t)u_{h}\ra dt\\
+\int_{0}^{T}\la U_{V}(t)u_{h},[a\left(  hD_{x}\right)  ,(V-V_{\epsilon}%
)\chi_{\epsilon}]U_{V}(t)u_{h}\ra dt\\
+\int_{0}^{T}\la U_{V}(t)u_{h},[a\left(  hD_{x}\right)
,V(1-\chi_{\epsilon })]U_{V}(t)u_{h}\ra dt.
\end{multline*}
For fixed $\epsilon$, the term $\int_{0}^{T}\la U_{V}(t)u_{h},[a\left(
hD_{x}\right)  ,V_{\epsilon}\chi_{\epsilon}]U_{V}(t)u_{h}\ra dt$ goes to $0$
as $h\To0$. The term $|\int_{0}^{T}\la U_{V}(t)u_{h},[a\left(  hD_{x}\right)
,(V-V_{\epsilon})\chi_{\epsilon}]U_{V}(t)u_{h}\ra dt|$ is less than
$2\epsilon\norm{a\left(  hD_{x}\right)}$. Finally,
\begin{multline*}
\left\vert \int_{0}^{T}\la U_{V}(t)u_{h},[a\left(  hD_{x}\right)
,V(1-\chi_{\epsilon})]U_{V}(t)u_{h}\ra dt\right\vert \\
\leq2\norm{V}_{\infty}\int_{0}^{T}\norm{a(hD_{x})U_V(t)u_h}_{L^{2}(\T^{d}%
)}\norm{(1- \chi_{\epsilon})U_V(t)u_h}_{L^{2}(\T^{d})}dt\\
\leq2\norm{V}_{\infty}\left(  \int_{0}^{T}\norm{a(hD_{x})U_V(t)u_h}_{L^{2}%
(\T^{d})}^{2}dt\right)  ^{1/2}\left(  \int_{0}^{T}%
\norm{(1- \chi_{\epsilon})U_V(t)u_h}_{L^{2}(\T^{d})}^{2}dt\right)  ^{1/2}%
\end{multline*}
and this goes to $0$ at the limits $h\To0$ and $\epsilon\To0$, by the same
argument as in \S \ref{s:linfty}. Again, we conclude that (\ref{e:limm}) holds
in this case. This concludes the proof of the proposition.
\end{proof}

\begin{corollary}
\label{c:vuLzero}Let $\Lambda$ be a primitive submodule of $\mathbb{Z}^{d}$.
If $\mu_{0}\left(  \mathbb{T}^{d}\times \Lambda^\bot\right)  =0$ then
$\sigma_{\Lambda}=0$, where $\sigma_{\Lambda}$ is the operator appearing in
Theorem \ref{t:precise}.
\end{corollary}

\subsection{End of proof of Theorem \ref{t:main}\label{s:martingale}}

Let us turn to the proof of the last assertion of Theorem \ref{t:main}. Let us
consider the disintegration of the limit measure $\mu$ with respect to $\xi$.
Here, to simplify the discussion, after normalizing $\mu$ we may
assume that it is a probability measure (this is no loss of generality, since
the result is trivially true when $\mu=0$). We call $\bar{\mu}$ the
probability measure on $\mathbb{R}^{d}$, image of $\mu(t,\cdot)$ under the
projection map $(x,\xi)\mapsto\xi$. We know that it does not depend on $t$. We
denote by $\mu_{\xi}(t,\cdot)$ the conditional law of $x$ knowing $\xi$, when
the pair $(x,\xi)$ is distributed according to $\mu(t,\cdot)$. Starting from Theorem \ref{t:main} (i), we now show that, for $\bar{\mu}$-almost every $\xi$, the probability measure
$\mu_{\xi}(t,\cdot)$ is absolutely continuous.

We consider a filtration, that is to say, a sequence $\mathcal{F}_{n}%
\subset\mathcal{F}_{n+1}$ of Borel $\sigma$-fields of $\mathbb{R}^{d}$, such
that $\cup_{n} \mathcal{F}_{n}$ generates the whole $\sigma$-field of Borel
sets. We will choose $\mathcal{F}_{n}$ generated by a finite partition made of
hypercubes (that is, a family of disjoint sets of the form $[a_{1},
b_{1})\times\ldots\times[a_{d}, b_{d})$, where $a_{d}<b_{d}$ can be finite or
infinite). For every $\xi$, there is a unique such hypercube containing $\xi$,
and we denote this hypercube by $\mathcal{F}_{n}(\xi)$. Finally, we choose
$\mathcal{F}_{n}$ such that $\bar\mu$ does not put any weight on the boundary
of each hypercube.

We know (by the martingale convergence theorem) that, for $\bar\mu$-almost
every $\xi$, for every continuous compactly supported function $f$ and every
non-negative integrable $\theta$,
\begin{equation}
\int_{\mathbb{T}^{d}}f(x, \xi)\mu_{\xi}(t, dx)\theta(t)dt=\lim_{n} \frac
{\int_{\mathbb{T}^{d}\times\mathcal{F}_{n}(\xi)}f(x, \eta)\mu(t,
dx,d\eta)\theta(t)dt}{\int\mu(t, \mathbb{T}^{d}\times\mathcal{F}_{n}%
(\xi))\theta(t)dt}. \label{martingale}%
\end{equation}
Fix $\xi$ such that \eqref{martingale} holds. Since $\mu(t,\cdot)$ is itself
the limit of the Wigner distributions $w_{h}(t,\cdot)$ and since it does not
put any weight on the boundary of $\mathcal{F}_{n}(\xi)$, we can
choose\newline-- a sequence of smooth compactly supported functions $\chi_{n}$
(obtained by convolution of the characteristic function of $\mathcal{F}%
_{n}(\xi)$ by a smooth kernel), and\newline-- a sequence $h_{n}$, going to
zero as fast as we wish,\newline such that
\begin{equation}
\int_{\mathbb{T}^{d}}f(x, \xi)\mu_{\xi}(t, dx)\theta(t)dt=\lim_{n} \frac
{\int_{\mathbb{T}^{d}\times\mathbb{R}^{d}}\chi_{n}^{2}(\eta)f(x, \eta
)w_{h_{n}}(t, dx,d\eta)\theta(t)dt}{\int_{\mathbb{T}^{d}\times\mathbb{R}^{d}}
\chi_{n}^{2}(\eta) w_{h_{n}}(t, dx,d\eta)\theta(t)dt}%
\end{equation}
for all smooth compactly supported $f$ and every $\theta$.

The absolute continuity of $\mu_{\xi}$ now follows from Theorem \ref{t:main} (i), applied to the sequence of functions
\[
v_{h_{n}}=\frac{\Op_{h_{n}}(\chi_{n})u_{h_{n}}}{|| \Op_{h_{n}}(\chi
_{n})u_{h_{n}}||}.
\]

\section{Observability estimates\label{s:obs}}

We now turn to the proof of Theorem \ref{t:obs}. Using the
uniqueness-compactness argument of Bardos, Lebeau and Rauch \cite{BLR} and a
Littlewood-Payley decomposition, one can reduce the proof of Theorem
\ref{t:obs} to the following Proposition \ref{p:obssemiclassic}. This is
clearly detailed in \cite{BurqZworski11}, from which we borrow the notation.
This reduction requires the potential to be {\em time-independent} and this is why
we make this assumption in Theorem \ref{t:obs}.

Let $\chi\in C_{c}^{\infty}\left(  \left(  -1/2,2\right)  \right)  $ be a
cut-off function equal to $1$ close to $1$ and define, for $h>0$:%
\[
\Pi_{h}u_{0}:=\chi\left(  h^{2}\left(  -\frac{1}{2}\Delta+V\right)  \right)
\]

\begin{proposition}
\label{p:obssemiclassic}Given any $T>0$ and any open set $\omega
\subset\mathbb{T}^{d}$, there exist $C,h_{0}>0$ such that:%
\begin{equation}
\left\Vert \Pi_{h}u_{0}\right\Vert _{L^{2}\left(  \mathbb{T}^{d}\right)  }%
^{2}\leq C\int_{0}^{T}\left\Vert U_{V}\left(  t\right)  \Pi_{h}u_{0}%
\right\Vert _{L^{2}\left(  \omega\right)  }^{2}dt, \label{e:obssc}%
\end{equation}
for every $0<h<h_{0}$ and every $u_{0}\in L^{2}\left(  \mathbb{T}^{d}\right)
$.
\end{proposition}

%\texttt{[This is the reason why it is necessary to impose that }$V$\texttt{
%does not depend on time. If }$V\left(  t,\cdot\right)  $\texttt{ actually
%depends on }$t$\texttt{ then we can no longer make sense of }$\Pi_{h}$\texttt{
%as we put it here -- we could try some averaged in time version, but I think
%the argument in Proposition 4.1 in \cite{BurqZworski11} would work in that
%context. It seems to me that the best way to prove the result for
%time-dependent potentials is to directly prove the inequality in Proposition
%4.1 of \cite{BurqZworski11}. But at this point we should not use semiclassical
%measures, but microlocal measures obtained by testing against operators
%}$Op_{1}\left(  a\right)  $\texttt{ with }$a$\texttt{ zero-homogeneous at
%infinity.]}

\begin{proof}
We argue by contradiction; if (\ref{e:obssc}) were false, then there would
exist a sequence $\left(  h_{n}\right)  $ tending to zero and $\left(
u_{0,n}\right)  $ in $L^{2}\left(  \mathbb{T}^{d}\right)  $ such that
$\Pi_{h_{n}}u_{0,n}=u_{0,n}$,%
\[
\left\Vert u_{0,n}\right\Vert _{L^{2}\left(  \mathbb{T}^{d}\right)
}=1\text{,\quad}\lim_{n\rightarrow\infty}\int_{0}^{T}\left\Vert U_{V}\left(
t\right)  u_{0,n}\right\Vert _{L^{2}\left(  \omega\right)  }^{2}dt=0.
\]
After eventually extracting a subsequence, we can assume that $\left(
u_{0,n}\right)  $ has a semiclassical measure $\mu_{0}$ and that the Wigner
distributions of $\left(  U_{V}\left(  t\right)  u_{0,n}\right)  $ converge
weak-$\ast$ to some $\mu\in L^{\infty}\left(  \mathbb{R};\mathcal{M}%
_{+}\left(  T^{\ast}\mathbb{T}^{d}\right)  \right)  $. By construction, we
have that:%
\[
\mu_{0}\left(  T^{\ast}\mathbb{T}^{d}\right)  =1,\quad\mu_{0}\left(
\mathbb{T}^{d}\times\left\{  0\right\}  \right)  =0;
\]
and therefore, by Proposition \ref{p:muc}, the same holds for $\mu\left(
t,\cdot\right)  $ for a.e. $t\in\mathbb{R}$. Moreover,%
\begin{equation}
\int_{0}^{T}\mu\left(  t,\omega\times\mathbb{R}^{d}\right)  dt=0\text{.}%
\label{e:obsmeas}%
\end{equation}

Now, we shall use Theorem \ref{t:precise} to obtain a contradiction. We first
establish the inequality for $d=1$ and then use an induction on the dimension.

\emph{Case }$d=1$. Since $\mu\left(  t,\mathbb{T}\times\left\{  0\right\}
\right)  =0$ and $\mu\left(  t,\mathbb{\cdot}\right)  $ is invariant by the
geodesic flow, it turns out that $\mu\left(  t,\cdot\right)  $ is constant.
Since (\ref{e:obsmeas}) holds, necessarily $\mu\left(  t,\mathbb{\cdot
}\right)  =0$, which contradicts the fact that $\mu\left(  t,T^{\ast
}\mathbb{T}\right)  =1$. This establishes Proposition \ref{p:obssemiclassic},
and therefore, Theorem \ref{t:obs} for $d=1$.

\emph{Case }$d\geq2$. We make the induction hypothesis that Proposition
\ref{p:obssemiclassic} holds for all tori $\R^{n}/2\pi\Gamma$ with $n\leq
d-1$, and $\Gamma$ a lattice in $\R^{n}$ such that $\left[ \la x, y\ra\in\IQ\:
\forall y\in\IQ\Gamma\Leftrightarrow x\in\IQ\Gamma\right] $.

Now, as shown in Theorem \ref{t:precise}, for $b\in L^{\infty}\left(
\mathbb{T}^{d}\right)  $ we have:%
\[
\int_{T^{\ast}\mathbb{T}^{d}}b\left(  x\right)  \mu\left(  t,dx,d\xi\right)
=\sum_{\Lambda}\int_{\T^{d}}b\left(  x\right)  \nu_{\Lambda}(t, dx)
=\sum_{\Lambda}\operatorname*{Tr}\left(  m_{\la b\ra_{\Lambda}}\,U_{\la V\ra_{\Lambda}%
}(t)\,\sigma_{\Lambda}\,{U_{\la V\ra_{\Lambda}}(t)}^{\ast}\right)  ,
\]
where $m_{\la b\ra_{\Lambda}}$ denotes multiplication by $ \la b\ra_{\Lambda}$ and $\sigma_{\Lambda}$ is a
trace-class positive operator on $L^{2}\left(  \mathbb{T}_\Lambda\right)  $,
where recall, $\mathbb{T}_{\Lambda}=\left\langle \Lambda\right\rangle
/2\pi\Lambda$.

For $\Lambda=0$, the measure $\nu_{\Lambda}(t)$ is constant in $x$, and since
$\nu_{\Lambda}(t, \omega)=0$ we have $\nu_{\Lambda}(t)=0$.

The fact that $\mu\left(  t,\mathbb{T}^{d}\times\left\{  0\right\}  \right)
=0$ implies that $\sigma_{\Lambda}=0$ for $\Lambda=\mathbb{Z}^{d}$. Therefore,
it suffices to show that $\sigma_{\Lambda}=0$ for every primitive non-zero
submodule $\Lambda\subset\mathbb{Z}^{d}$ of rank $\leq d-1$.

The torus $\mathbb{T}_{\Lambda}$ has dimension $\leq d-1$ and falls into the
range of our induction hypothesis. Since (\ref{e:obsmeas}) holds, we conclude
that:%
\[
\int_{0}^{T}\operatorname*{Tr}\left(  m_{\la\mathbf{1}_{\omega}\ra_{\Lambda}}\,U_{\la
V\ra_{\Lambda}}(t)\,\sigma_{\Lambda}\,{U_{\la V\ra_{\Lambda}}(t)}^{\ast
}\right)  dt=0,
\]
and hence
\[
\int_{0}^{T}\operatorname*{Tr}\left(  m_{\mathbf{1}_{\la\omega\ra_{\Lambda}}}\,U_{\la
V\ra_{\Lambda}}(t)\,\sigma_{\Lambda}\,{U_{\la V\ra_{\Lambda}}(t)}^{\ast
}\right)  dt=0,
\]
where $\la\omega\ra_{\Lambda}$ is the open set where $\la\mathbf{1}_{\omega}\ra_{\Lambda}>0$. By our induction hypothesis we have:\footnote{To deduce
this from Theorem \ref{t:obs}, it suffices to write $\sigma_{\Lambda}$ as a
linear combination of orthogonal projectors on an orthonormal basis of
eigenfunctions of $\sigma_{\Lambda}$:%
\[
\sigma_{\Lambda}=\sum_{n\in\mathbb{N}}\lambda_{n}\left\vert \phi
_{n}\right\rangle \left\langle \phi_{n}\right\vert ;
\]
since $\lambda_{n}\geq0$ and
$\sum_{n\in\mathbb{N}}\lambda_{n}<\infty$ the observability
inequality for $\sigma_{\Lambda}$ follows from the fact that it
holds for every $\phi_{n}$.}%
\[
\operatorname*{Tr}\left(  \sigma_{\Lambda}\right)  \leq C(T,\la\omega\ra_{\Lambda}) \int_{0}%
^{T}\operatorname*{Tr}\left(  m_{\mathbf{1}_{\la\omega\ra_{\Lambda}}}\,U_{\la V\ra_{\Lambda}%
}(t)\,\sigma_{\Lambda}\,{U_{\la V\ra_{\Lambda}}(t)}^{\ast}\right)  dt,
\]
and thus $\sigma_{\Lambda}=0$ (for all $\Lambda$) and $\mu(t,T^{\ast}%
\T^{d})=0$. This contradicts the fact that $\mu\left(  t,T^{\ast}%
\mathbb{T}^{d}\right)  =1$.
\end{proof}

Coming back to the semiclassical measures of Theorem \ref{t:main}, it is now obvious that
$$\int_0^T \mu(t, \omega\times\R^d)dt\geq \frac{T}{C(T, \omega)}  \mu_0(T^*\T^d ).$$
Corollary \ref{c:positiveopen} can then be derived by the same argument as in \S \ref{s:martingale}.

\section{Appendix~: pseudodifferential calculus\label{s:PDO}}

In the paper, we use the Weyl quantization with parameter $h$,
that associates to a function $a$ on
$T^{*}\IR^{d}=\IR^{d}\times\IR^{d}$ an operator $\Op_{h}(a)$, with
kernel
\[
K_{a}^{h}(x, y)=\frac{1}{(2\pi h)^{d}}\int_{\IR^{d}} a\left(
\frac{x+y}2, \xi\right)  e^{\frac{i}h \xi.(x-y)}d\xi.
\]
If $a$ is smooth and has uniformly bounded derivatives, then this
defines a continuous operator $\cS(\IR^{d})\To \cS(\IR^{d})$, and
also $\cS^{\prime }(\IR^{d})\To \cS^{\prime}(\IR^{d})$. If $a$ is
$(2\pi\IZ)^{d}$-periodic with respect to the first variable (which
is always the case in this paper), the operator preserves the
space of $(2\pi\IZ)^{d}$-periodic distributions on $\IR^{d}$. We
note the relation $\Op_{h}(a(x, \xi))=\Op_{1}(a(x, h\xi)).$

We use two standard results of pseudodifferential calculus.

\begin{theorem}
(The Calder\'on-Vaillancourt theorem)

There exists an integer $K_{d}$, and a constant $C_{d}>0$
(depending on the dimension $d$) such that, if $a$ if a smooth
function on $T^{*}\IT^{d}$, with uniformly bounded derivatives,
then
\[
\norm{\Op_1(a)}_{L^{2}(\IT^{d})\To L^{2}(\IT^{d})}\leq
C_{d}\sum_{\alpha \in\IN^{2d},|\alpha|\leq K_{d}}
\sup_{T^{*}\IT^{d}}|\partial^{\alpha}a|.
\]

\end{theorem}

A proof in the case of $L^{2}(\IR^{d})$ can be found in
\cite{DimassiSjostrand}. It can be adapted to the case of a
compact manifold by working locally, in coordinate charts.

We also recall the following formula for the product of two
pseudodifferential
operators (see for instance \cite{DimassiSjostrand}, p. 79)~: $\Op_{1}%
(a)\circ\Op_{1}(b)=\Op_{1}(a\sharp b),$ where
\[
a\sharp b(x, \xi)=\frac1{(2\pi)^{4d}}\int_{\IR^{4d}}
e^{\frac{i}2\sigma(u_{1}, u_{2})} (\cF a_{z})(u_{1})(\cF
b_{z})(u_{2}) du_{1}du_{2},
\]
where we let $z=(x, \xi)\in\IR^{2d}$, $a_{z}$ is the function
$\omega\mapsto a(z+\omega)$, and $\cF$ is the Fourier transform.
We can deduce from this formula and from the
Calder\'on-Vaillancourt theorem the following estimate~:

\begin{proposition}
Let $a$ and $b$ be two smooth functions on $T^{*}\IT^{d}$, with
uniformly bounded derivatives.
\[
\norm{\Op_1(a)\circ \Op_1(b)-\Op_1(ab)}_{L^{2}(\IT^{d})\To
L^{2}(\IT^{d})}\leq C_{d} \sum_{\alpha\in\IN^{2d},|\alpha|\leq
K_{d}} \sup_{T^{*}\IT^{d}} |\partial^{\alpha}D(a, b)|,
\]
where we denote $D(a, b)$ the function $D(a, b)(x, \xi)=\left(
\partial _{x}\partial_{\eta}-\partial_{y}\partial_{\xi}\right)
(a(x, \xi)b(y, \eta))\rceil_{x=y, \eta=\xi}.$
\end{proposition}

We finally deduce the following corollary. We use the notations of
Section \ref{s:second}.

\begin{corollary}
Let $a\in C^{\infty}(\IT^{d}\times\IR^{d})$ have uniformly bounded
derivatives, and let $\chi\in C_{c}^{\infty}(\IR^{d})$ be a
nonnegative cut-off function such that $\sqrt{\chi}$ is smooth.
Let $0<h<1$ and $R>1$. Denote
\[
a_{R}(x,\xi)=a(x,\xi)\chi\left(  \frac{P_{\Lambda}\xi}{hR}\right)
.
\]
Assume that $a>0$, and denote $b_{R}=\sqrt{a_{R}},$ Then
\[
\norm{\Op_h(a_R)-\Op_h(b_R)^2}_{L^{2}(\IT^{d})\To L^{2}(\IT^{d})}%
=\cO(h)+\cO(R^{-1})
\]
in the limits $h\To0$ and $R\To+\infty.$
\end{corollary}

\begin{corollary}
\label{CoroCV} Let $a\in
C^{\infty}(\IT^{d}\times\IR^{d}\times\IR^{d})$, $0$-homogeneous in
the third variable outside a compact set, with uniformly bounded
derivatives, and let $\chi\in C_{c}^{\infty}(\IR^{d})$ be a
nonnegative cut-off function such that $\sqrt{\chi}$ is smooth.
Let $0<h<1$ and $R>1$. Denote
\[
a^{R}(x,\xi)=a\left(  x,\xi,\frac{P_{\Lambda}\xi}{h}\right)
\left( 1-\chi\left(  \frac{P_{\Lambda}\xi}{hR}\right)  \right)  .
\]
Assume that $a>0$, and denote $b^{R}=\sqrt{a^{R}}.$ Then
\[
\norm{\Op_h(a^R)-\Op_h(b^R)^2}_{L^{2}(\IT^{d})\To
L^{2}(\IT^{d})}=\cO(R^{-1})
\]
in the limits $h\To0$ and $R\To+\infty.$
\end{corollary}

%\input{intro.tex}
%\input{invariant.tex}
%\input{2micro.tex}
%\input{recursion.tex}
%\input{propagation.tex}
%\input{endthm1}
%\input{obs.tex}
%\input{appendix.tex}

%\bibliographystyle{plain}
%\bibliography{v:/biblio}

%\input{InvariantMeasuresPotential.bbl}
\def\cprime{$'$}

\end{document}